\documentclass[11pt,reqno]{amsart}
\usepackage[margin=1in]{geometry}
\usepackage[utf8]{inputenc}
\usepackage[T1]{fontenc}
\usepackage{amssymb}
\usepackage{graphicx} 
\usepackage{mathrsfs}
\usepackage{enumerate}
\usepackage{xspace}
\usepackage{color}
\usepackage{enumerate}
\usepackage{enumitem}
\usepackage{amsthm}
\usepackage{dsfont}
\usepackage{bbm}
\usepackage{commath}
\usepackage[colorlinks=true,linkcolor=blue]{hyperref}
\usepackage[numbers]{natbib}
\usepackage[backgroundcolor=white,bordercolor=red]{todonotes}
\DeclareMathAlphabet{\mathpzc}{OT1}{pzc}{m}{it}
\usepackage[nameinlink]{cleveref}
\numberwithin{equation}{section}
\begin{document}

\expandafter\let\expandafter\oldproof\csname\string\proof\endcsname
\let\oldendproof\endproof
\renewenvironment{proof}[1][\proofname]{%
	\oldproof[\scshape\hspace{1em}#1]%
}{\oldendproof}

\newtheoremstyle{mystyle_thm}% ⟨name ⟩ 
{6pt}% ⟨Space above ⟩
{6pt}% ⟨Space below ⟩
{\itshape}% ⟨Body font ⟩
{1em}% ⟨Indent amount ⟩2
{\scshape}% ⟨Theorem head font⟩
{.}% ⟨Punctuation after theorem head ⟩
{.5em}% ⟨Space after theorem head ⟩3
{}%

\newtheoremstyle{mystyle_def}% ⟨name ⟩ 
{6pt}% ⟨Space above ⟩
{6pt}% ⟨Space below ⟩
{}% ⟨Body font ⟩
{1em}% ⟨Indent amount ⟩2
{\scshape}% ⟨Theorem head font⟩
{.}% ⟨Punctuation after theorem head ⟩
{.5em}% ⟨Space after theorem head ⟩3
{}%

\theoremstyle{mystyle_thm}

\newtheorem{theorem}{Theorem}[section]
\newtheorem{lemma}[theorem]{Lemma}
\newtheorem{proposition}[theorem]{Proposition}
\newtheorem{corollary}[theorem]{Corollary}
\newtheorem{definition}[theorem]{Definition}
\newtheorem{Ass}[theorem]{Assumption}
\newtheorem{condition}[theorem]{Condition}
\newtheorem*{lemma*}{Lemma}

\theoremstyle{mystyle_def}

\newtheorem{example}[theorem]{Example}
\newtheorem{remark}[theorem]{Remark}
\newtheorem{SA}[theorem]{Standing Assumption}
\newtheorem{discussion}[theorem]{Discussion}
\newtheorem{remarks}[theorem]{Remark}
\newtheorem*{notation}{Remark on Notation}
\newtheorem{application}[theorem]{Application}

%Stochastic Intervals
\newcommand{\of}{[\hspace{-0.06cm}[}
\newcommand{\gs}{]\hspace{-0.06cm}]}

%Lebesgue Measure
\newcommand{\llambda}{{\mathchoice
		{\lambda\mkern-4.5mu{\raisebox{.4ex}{\scriptsize$\setminus$}}}
		{\lambda\mkern-4.83mu{\raisebox{.4ex}{\scriptsize$\setminus$}}}
		{\lambda\mkern-4.5mu{\raisebox{.2ex}{\footnotesize$\scriptscriptstyle\setminus$}}}
		{\lambda\mkern-5.0mu{\raisebox{.2ex}{\tiny$\scriptscriptstyle\setminus$}}}}}

%Indikator
\newcommand{\1}{\mathds{1}}

%Filtrations
\newcommand{\F}{\mathbf{F}}
\newcommand{\G}{\mathbb{G}}

\newcommand{\B}{\mathbf{B}}

%Martingale Measures
\newcommand{\M}{\mathcal{M}}

%Scalar Product
\newcommand{\la}{\langle}
\newcommand{\ra}{\rangle}

%Special Type of Quadratic Variation
\newcommand{\lle}{\langle\hspace{-0.085cm}\langle}
\newcommand{\rre}{\rangle\hspace{-0.085cm}\rangle}
\newcommand{\blle}{\Big\langle\hspace{-0.155cm}\Big\langle}
\newcommand{\brre}{\Big\rangle\hspace{-0.155cm}\Big\rangle}

%Coordinate Process
\newcommand{\X}{\mathsf{X}}

\newcommand{\lv}{I^-}
\newcommand{\uv}{I^+}

%Short Cuts
\newcommand{\tr}{\operatorname{tr}}
\newcommand{\N}{{\mathbb{N}}}
\newcommand{\cadlag}{c\`adl\`ag }
\newcommand{\on}{\operatorname}
\newcommand{\oP}{\overline{P}}
\newcommand{\oO}{\mathcal{O}}
\newcommand{\D}{\mathsf{D}} %{D(\mathbb{R}_+; \mathbb{R})}
\newcommand{\bx}{\mathsf{x}}
\newcommand{\bb}{\hat{b}}
\newcommand{\bs}{\hat{\sigma}}
\newcommand{\bv}{\hat{v}}
\renewcommand{\v}{\mathfrak{m}}
\newcommand{\ob}{\bar{b}}
\newcommand{\oa}{\bar{a}}
\newcommand{\os}{\widehat{\sigma}}
\renewcommand{\j}{\varkappa}
\newcommand{\scl}{\ell}
\newcommand{\Y}{\mathscr{Y}}
\newcommand{\Z}{\mathscr{Z}}
\newcommand{\T}{\mathcal{T}}
\newcommand{\con}{\mathsf{c}}
\newcommand{\nk}{\hspace{-0.25cm}{{\phantom A}_k^n}}
\newcommand{\nl}{\hspace{-0.25cm}{{\phantom A}_1^n}}
\newcommand{\nm}{\hspace{-0.25cm}{{\phantom A}_2^n}}
\newcommand{\n}{\hspace{-0.35cm}{\phantom {Y_s}}^n}
\newcommand{\nme}{\hspace{-0.35cm}{\phantom {Y_s}}^{n - 1}}
\renewcommand{\o}{\hspace{-0.35cm}\phantom {Y_s}^0}
\newcommand{\e}{\hspace{-0.4cm}\phantom {U_s}^1}
\newcommand{\z}{\hspace{-0.4cm}\phantom {U_s}^2}
\newcommand{\iii}{|\hspace{-0.05cm}|\hspace{-0.05cm}|}
\newcommand{\co}{\overline{\on{co}}}
\newcommand{\ovb}{\overline{b}}
\newcommand{\ova}{\overline{a}}
\newcommand{\s}{\mathfrak{s}}
\newcommand{\opsi}{\overline{\Psi}}
\newcommand{\ol}{\mathcal{L}}
\newcommand{\cW}{\mathscr{W}}
\newcommand{\cU}{\mathcal{U}}
\newcommand{\oD}{\overline{D}}
\newcommand{\ua}{\underline{a}}
\newcommand{\ou}{\overline{b}}
\newcommand{\uu}{\underline{b}}
\renewcommand{\d}{\text{\rm d}}

%Typographical
\renewcommand{\epsilon}{\varepsilon}
%\renewcommand{\rho}{\varrho}

%Semimartingale laws
\newcommand{\fPs}{\fP_{\textup{sem}}}
\newcommand{\fPas}{\mathfrak{S}_{\textup{ac}}}
\newcommand{\rrarrow}{\twoheadrightarrow}
\newcommand{\cA}{\mathcal{A}}
\newcommand{\ocA}{\mathcal{U}}
\newcommand{\cR}{\mathcal{K}}
\newcommand{\cK}{\mathcal{K}}
\newcommand{\cQ}{\mathcal{Q}}
\newcommand{\cF}{\mathcal{F}}
\newcommand{\cE}{\mathcal{E}}
\newcommand{\cC}{\mathcal{C}}
\newcommand{\cD}{\mathcal{D}}
\newcommand{\bC}{\mathbb{C}}
\newcommand{\cH}{\mathcal{H}}
\newcommand{\bth}{\overset{\leftarrow}\theta}
\renewcommand{\th}{\theta}
\newcommand{\cG}{\mathcal{G}}
\newcommand{\fPasn}{\mathfrak{S}^{\textup{ac}, n}_{\textup{sem}}}
\newcommand{\CLM}{\mathfrak{M}^\textup{ac}_\textup{loc}}
\newcommand{\Sd}{\mathcal{S}^\textup{sp}_{\textup{d}}}
\newcommand{\Sc}{\mathcal{S}}
\newcommand{\Sac}{\mathcal{S}_\textup{ac}}
\newcommand{\A}{\mathbb{A}}
\newcommand{\Td}{\mathsf{T}^\textup{d}}
\renewcommand{\t}{\mathfrak{t}}
%\DeclareMathOperator*{\gr}{gr}

%Basic Stuff
\newcommand{\bR}{\mathbb{R}}
\newcommand{\nnabla}{\nabla}
\newcommand{\f}{\mathfrak{f}}
\newcommand{\g}{\mathfrak{g}}
\newcommand{\oconv}{\overline{\on{co}}\hspace{0.075cm}}
\renewcommand{\a}{\mathfrak{a}}
\renewcommand{\b}{\mathfrak{b}}
\renewcommand{\d}{\mathsf{d}}
\newcommand{\bS}{\mathbb{S}^d_+}
\newcommand{\p}{\mathsf{p}}
\newcommand{\dr}{r} %{\mathsf{r}}
\newcommand{\m}{\mathbb{M}}
\newcommand{\Q}{Q}
\newcommand{\usc}{\textit{USC}}
\newcommand{\lsc}{\textit{LSC}}
\newcommand{\q}{\mathfrak{q}}
\renewcommand{\X}{\mathbb{X}}
\newcommand{\W}{\mathscr{W}}
\newcommand{\fP}{\mathcal{P}}
\newcommand{\w}{\mathsf{w}}
\newcommand{\oM}{\mathsf{M}}
\newcommand{\oZ}{\mathsf{Z}}
\newcommand{\oK}{\mathsf{K}}
\renewcommand{\Re}{\operatorname{Re}}
\newcommand{\cCk}{\mathsf{c}_k}
\newcommand{\C}{\mathsf{C}}
\newcommand{\cP}{\mathcal{P}}
\newcommand{\oPi}{\overline{\Pi}}
\newcommand{\cI}{\mathcal{I}}
\renewcommand{\P}{\mathbb{P}}
\newcommand{\E}{\mathbb{E}}

\newcommand{\UC}{\textit{UC}}
\renewcommand{\N}{\mathcal{C}}
\renewcommand{\M}{\mathcal{H}}
\renewcommand{\G}{\mathbb{S}}
\renewcommand{\L}{{\mathcal{L}_\eta^2}}

\renewcommand{\emptyset}{\varnothing}

\allowdisplaybreaks

\makeatletter
%\@namedef{subjclassname@2020}{%
%	\textup{2020} Mathematics Subject Classification} Primary 46N10; 60G65; 93E20 Secondary 47D07; 49J53
%\makeatother

 \title[Small-Noise Analysis of Controlled Functional Differential Equations]{A Small-Noise Analysis of Controlled Functional Differential Equations with Gaussian Noise} 
 %\\
\author[D. Criens]{David Criens}
\address{University of Freiburg, Ernst-Zermelo-Str. 1, 79104 Freiburg, Germany.}
\email{david.criens@stochastik.uni-freiburg.de}

\author[M. Nendel]{Max Nendel}
\address{University of Waterloo, 200 University Ave W, N2L 3G1, Waterloo, Ontario, Canada.}
\email{mnendel@uwaterloo.ca}

\thanks{Financial support through the
Natural Sciences and Engineering Research Council of Canada via Discovery Grant no.\ RGPIN-2025-04219 is gratefully acknowledged.}
\date{\today}

\begin{abstract}
We study small-noise asymptotics for controlled functional differential equations driven by additive Gaussian noise.\ The Gaussian noise is modeled on an abstract Wiener space, covering both classical Brownian perturbations and non-Markovian perturbations such as fractional Brownian motion.\ The drift coefficient is assumed to be non-anticipative, Lipschitz continuous in the state path, and of linear growth.\ For bounded uniformly continuous cost functionals, we prove game-theoretic lower and upper bounds for the small-noise logarithmic value functions and identify their limit whenever the associated deterministic zero-sum game has a value.\ In the limiting game, one player chooses the drift control, while the other selects a Cameron--Martin shift of the Gaussian noise, penalized by the corresponding quadratic energy cost. We further provide sufficient Fan-type convexity and concavity conditions under which the game has a value, thereby obtaining a full small-noise Laplace principle.\ The proof combines the Boué--Dupuis variational representation on abstract Wiener spaces with pathwise stability of the controlled solution map and adapted finite-dimensional approximations of Cameron--Martin shifts.

   
\end{abstract}

\maketitle

\section{Introduction}

Small random perturbations of controlled dynamical systems lead naturally to risk-sensitive control problems. In the classical Brownian setting, consider
\[
d Y^{\varepsilon, \lambda}_t = b (Y^{\varepsilon, \lambda}_t, \lambda_t) \, dt + \sqrt{\varepsilon} \, d W_t, \quad \text{where }W \text{ is a Brownian motion}.
\]  
A central question is to identify 
\begin{align} \label{eq: intro motivating limit}
\varepsilon \inf_{\lambda} \log \mathbb{E} \Big[ e^{\varphi (Y^{\varepsilon, \lambda}_T) / \varepsilon} \Big] \quad \text{as } \varepsilon \downarrow 0.
\end{align} 
Such limits are closely related to the Freidlin–Wentzell theory \cite{FreidlinWentzell2012} and are typically represented by deterministic zero-sum games in which an additional player selects an adverse perturbation and pays a quadratic energy cost, see \cite{BD_01, FleSon_06, J_92}. 

The present paper develops a version of this principle for controlled functional differential equations with additive Gaussian noise, where the coefficients may depend on the past history of the state and the noise may have long memory.
The stochastic system is formally written as
\begin{equation}\label{eq:intro-stochastic-system}
    dY_t^{\varepsilon, \lambda}
    =\mu(t,Y^{\varepsilon, \lambda},\lambda_t)\,dt+
      \sqrt{\varepsilon}\,d\X_t,
    \qquad
    Y_0^\varepsilon=x_0,
\end{equation}
where $\lambda$ is a progressively measurable control, \(\mu\) is a Lipschitz continuous non-anticipative coefficient, and $\X$ is the coordinate process under a centered Gaussian measure on a path space, introduced through an abstract Wiener space $(\Omega,\M,\mathbb P)$, where $\M$ denotes the corresponding Cameron--Martin space.  
This framework contains standard Brownian motion, but also fractional Brownian motion and other Gaussian processes with memory.  In particular, when $\X$ is a fractional Brownian motion with Hurst parameter $H\in(0,1)$, the relevant Cameron--Martin space is the image of $L^2$ under the fractional Volterra kernel, see \cite{Budhiraja2025, DU_99,Mishura2008,Nualart2006}. 

The main object of interest is the logarithmic value
\begin{equation}\label{eq:intro-value}
    V^\varepsilon
    :=
    \varepsilon\inf_{a\in\A}
    \log \mathbb E\bigg[
       \exp\bigg(
          \frac{\varphi\big(\G(a,\sqrt{\varepsilon}\X)\big)}{\varepsilon}
       \bigg)
    \bigg],
\end{equation}
where $\G$ is the pathwise solution map associated with the deterministic controlled equation, i.e., 
\begin{equation}\label{eq:intro-solution-map}
    \G_t(v,w)
    =x_0+
      \int_0^t \mu(s,\G(v,w),v_s)\,ds + w_t
    \quad \text{for } t\in[0,1].
\end{equation}
Here, $a$ ranges over adapted control strategies and $\varphi$ is a bounded uniformly continuous cost on the path space.\ Our main result identifies the \(\varepsilon \downarrow 0\) limit of \eqref{eq:intro-value} as the value of the deterministic zero-sum game with cost function
\begin{equation}\label{eq:intro-game-cost}
    C(v,u)
    :=\varphi\big(\G(v,u)\big)-\frac12\|u\|_{\M}^2
    \quad \text{for }v\in\N\text{ and } u\in\M, 
\end{equation}
where \(\N\) is a space of control strategies.\
More precisely, if $\lv$ and $\uv$ denote the lower and upper values of the corresponding non-anticipative deterministic game, we show that 
\begin{equation}\label{eq:intro-main-result}
    \lv \leq \liminf_{\varepsilon \downarrow 0} V^\varepsilon \leq \limsup_{\varepsilon \downarrow 0} V^\varepsilon \leq \uv.
\end{equation}
In particular, if the game is fair in the sense that \(\lv = \uv\), then 
\[
\lim_{\varepsilon\downarrow0}V^\varepsilon = \lv = \uv.
\]
This result may be viewed as a path-dependent, non-Markovian analogue of classical small-noise risk-sensitivity results for control problems, where in the limiting game an artificial player is constrained by the geometry of the Cameron--Martin space of the driving Gaussian process.

We also derive a concrete sufficient condition for the fairness of the game based on Fan’s minimax theorem. Namely, we show that the game is fair if the cost \(C (v, u)\) is convex in \(v\) and concave in \(u\), referring to Fan's extended definitions of convexity and concavity.\ Moreover, we discuss viscosity methods that imply the fairness of the game.

A key feature of our approach is that it avoids stochastic calculus for the driving noise.  Since the noise enters additively, the equation is interpreted through the deterministic solution map \eqref{eq:intro-solution-map}; pathwise existence, uniqueness, and stability follow from Picard iteration and Gronwall estimates.  This is particularly useful for fractional Brownian motion, which is not a semimartingale except for the standard Brownian case. Thus, stochastic integration with respect to it usually requires Young, rough-path, Skorokhod, or Malliavin techniques depending on the value of the Hurst parameter and on the coefficient structure.\ In the present additive setting, none of these constructions is needed.\ The memory of the noise enters only through the Cameron--Martin norm and the adaptedness structure.

The proof has three main ingredients.  First, we use the Bou\'e--Dupuis variational representation for exponential functionals of Gaussian processes.\ The original formula was established for Brownian motion in \cite{BD_98}, extensions and systematic developments of the weak-convergence approach to large deviations were developed in \cite{BudhirajaDupuis2000,BudhirajaDupuis2019,dupuisellis}, while the abstract Wiener space version used here is due to Zhang \cite{Z_09}.\  Second, we exploit the continuity properties of the deterministic solution map $\G$, which transfer small perturbations of the driving path into small perturbations of the controlled state uniformly over controls.\  Third, we construct adapted finite-dimensional approximations of Cameron--Martin shifts.\ In the Volterra case this construction is particularly transparent: one approximates an $L^2$ input by block averages and shifts the result one time step forward, thereby preserving causality.\ In the general abstract Wiener setting, the same mechanism can be formulated through a continuous resolution of the identity and the spectral multiplicity theorem for projection-valued measures.

The use of adapted approximations is one of the technical points that distinguishes the present argument from a direct application of Schilder's theorem.  A compact embedding $\M\hookrightarrow\Omega$ gives finite-rank approximation in principle, but adaptedness requires the approximants to read only the past and write into the future.  This is why the proof uses delayed finite-rank contractions rather than ordinary orthogonal projections.  The construction is closely related in spirit to adapted Wong--Zakai approximations and to the approximation procedures that appear in support theorems and weak-convergence proofs of large deviations.

The paper also provides explicit examples.\ For a controlled fractional-kernel model with terminal cost, the limiting game can be reduced to a two-dimensional static optimization involving the accumulated control and the terminal Cameron--Martin displacement. The value depends explicitly on the Hurst parameter~\(H\) through the variance $T^{2H}$ at the observation time $T$.  This example illustrates that the limit detects the memory structure of the perturbation even when the controlled dynamics themselves are elementary.  From a modeling perspective, this is relevant for systems driven by colored or persistent noise, such as viscoelastic materials, anomalous transport, hydrological and network traffic models, or controlled systems subject to slowly decaying environmental correlations.\ In these settings, the cost of producing a rare displacement is not determined by a Brownian energy but by the Cameron--Martin geometry of the underlying correlated Gaussian field.

We now place the contribution in the surrounding literature.\ The foundational results on small-noise asymptotics for stochastic differential equations go back to Freidlin and Wentzell \cite{FreidlinWentzell2012}, see \cite{DZ_98,DeuschelStroock1989, dupuisellis, Varadhan1984} for general large deviations methods and \cite{zbMATH06864073,zbMATH07370698} for recent developments in general settings.\ A key tool in the theory is the variational representation of the logarithmic transform, which are precisely the dual representations of entropic risk measures, cf.\ \cite{FS2016}.\ Its applications to large deviations are a corner stone of the weak-convergence approach developed in \cite{BudhirajaDupuis2000,BudhirajaDupuis2019, dupuisellis}.\ Variational representations of Bou\'e--Dupuis-type have proved especially effective because they convert exponential asymptotics into stochastic control problems, see \cite{BVLT_20, BD_98,Budhiraja2024, CK_25, Z_09}.\ 
Large deviations and Laplace principles for fractional Brownian motion and its functionals rely on the corresponding Cameron--Martin structure and have been studied from several viewpoints, see, among others, \cite{Budhiraja2025,DU_99,Nualart2006,Mishura2008}.\ Functional and delay equations driven by fractional or Volterra noise have also received sustained attention, especially in connection with pathwise integration, rough paths, and memory effects, see \cite{FrizVictoir2010, NualartRascanu2002}.\ Recent developments related to large deviations can, for example, be found in \cite{Budhiraja2025, JP_22, MRTZ_16, NR_00}. 
The present paper differs from these works by combining path-dependent controlled dynamics, additive abstract Gaussian perturbations, and a game-theoretic Laplace limit in a single framework.

The formulation is intended to be flexible enough for both stochastic control and applied probability.\ For readers interested in stochastic control, the theorem identifies the small-noise limit of an exponential risk-sensitive criterion as a deterministic non-anticipative game.  For readers interested in large deviations, it gives a Laplace principle in which the contraction map itself depends on a minimizing control and where the limiting variational problem retains the temporal information encoded in the Cameron--Martin space.  For readers interested in fractional or colored noise models, the result gives a way to handle memory without first developing an It\^o theory for the noise.

The remainder of the paper is organized as follows.\  Section~\ref{sec:main-results} introduces the abstract Wiener framework, the controlled functional differential equation, and the deterministic game.\ It then states the main theorem, Theorem~\ref{theo: main zusammengesetzt}, and discusses examples, including a fractional-kernel game whose value depends explicitly on the Hurst parameter. Section~\ref{sec:proof} presents the proof of Theorem~\ref{theo: main zusammengesetzt}.\ We first develop the required functional-analytic approximation tools, cf.\ Section~\ref{sec:funkana}, then establish the Bou\'e--Dupuis representation in Section \ref{sec:boue.dupuis}, and finally prove the lower and upper bounds, see Section \ref{sec:first.ineq} and Section \ref{sec:last.ineq}, respectively.

\section{The Setup and Main Result}\label{sec:main-results}

For a fixed dimension \(d \in \mathbb{N}\), let \(\Omega := C_0 ([0, 1]; \bR^d)\) be the space of continuous functions from \([0, 1]\) to \(\bR^d\) that start in the origin, endowed with the uniform topology induced by the norm \(\| \omega \|_\infty := \| \omega \|_1\), where \(\| \omega \|_t := \sup_{s \in [0, t]} \| \omega (s) \|\)  for $t\in [0,1]$.\ Let \(\cF := \mathcal{B} ( \Omega )\) be the corresponding Borel \(\sigma\)-field.\ The coordinate process is denoted by \(\X_t (\omega) = \omega (t)\) for \(\omega \in \Omega\) and \(t \in [0, 1]\).\ Moreover, let \(\Lambda\) be a nonempty Polish space, which serves as an action space for the control problem under consideration.\ The coefficient is given by the following standing assumption. 

\begin{SA} \label{SA: main coefficient}
	Let \(\mu \colon [0, 1] \times C ([0, 1]; \bR^d) \times \Lambda \to \bR^d\) be a Borel map\footnote{Throughout, \(C ([0, 1]; \bR^d)\) is the space of continuous functions from \([0, 1]\) to \(\bR^d\) endowed with the uniform topology.} with the following properties:
	\begin{enumerate}
		\item[(A1)] \(\mu\) is non-anticipative in the sense that, for all $t \in [0, 1]$, $\lambda \in \Lambda$, and $\alpha, \alpha' \in C ([0, 1]; \bR^d)$ with $\alpha = \alpha'$ on $[0, t]$, also $\mu (t, \alpha, \lambda) = \mu (t, \alpha', \lambda)$.
		\item[(A2)] \(\mu\) satisfies the following Lipschitz and linear growth assumptions: there is a constant \(L > 0\) such that 
		\begin{align*}
			\| \mu (t, \alpha, \lambda) - \mu (t, \alpha', \lambda) \| \leq L \, \| \alpha - \alpha' \|_t\quad \text{and}\quad
			\| \mu (t, \alpha, \lambda) \| \leq L \, \big( 1 + \| \alpha \|_t \big)
		\end{align*}
		for all $t \in [0, 1]$, $\lambda \in \Lambda$, and $\alpha, \alpha' \in C ([0, 1]; \bR^d)$.
        \item[(A3)] The map $\lambda \mapsto \mu (t, \alpha, \lambda)$ is uniformly equi-continuous, i.e., there exists a modulus of continuity $m_\mu$ such that, for all $t \in [0, 1]$, $\lambda, \lambda' \in \Lambda$, and $\alpha \in C ([0, 1]; \bR^d)$,
        \[
        \| \mu (t, \alpha, \lambda) - \mu (t, \alpha, \lambda') \| \leq m_\mu \big( d_\Lambda (\lambda, \lambda') \big), 
        \]
        where $d_\Lambda$ denotes a complete metric on $\Lambda$ that induces its topology.
	\end{enumerate}
\end{SA}

We are interested in the influence of random Gaussian perturbations to the deterministic controlled system 
\begin{align} \label{eq: deterministic controlled equation}
d Y_t = \mu (t, Y, \lambda_t) \, dt, \quad Y_0 = x_0, 
\end{align} 
where $x_0 \in \bR^d$ is an initial value that remains fixed throughout the paper.\ In the following, we introduce a general framework for Gaussian noise that includes classical white noise as well as Gaussian perturbations with memory, coming, for example, from fractional Brownian motion (fBM). 

	Let \(\P\) be a non-degenerate centered Gaussian measure on \((\Omega, \cF)\), i.e., for every non-zero \(\gamma\) in the topological dual $\Omega^*$ of the Banach space $(\Omega, \| \cdot \|_\infty)$, the random variable $\Omega \ni \omega \mapsto {_\Omega}\langle \omega, \gamma \rangle_{\Omega^*}$ is a centered Gaussian random variable with non-zero variance.\ The expectation operator corresponding to the probability measure \(\P\) is denoted by \(\E\).
	
	It is well-known (see, e.g., \cite[Theorem~8.2.3]{stroock_APV_V3}) that there exists a unique separable Hilbert space $(\M, \| \cdot \|_{\M}, \langle \cdot, \cdot \rangle_{\M})$ that is densely and continuously embedded as a subspace of $(\Omega, \| \cdot \|_\infty)$ such that 
	\[
	\E \Big[ e^{ i\, {_\Omega}\langle \, \cdot \, , \, \gamma \rangle_{\Omega^*}} \Big] = e^{- \frac{1}{2} \| h_\gamma \|_{\M}^2 } \quad \text{for all }\gamma \in \Omega^*, 
	\] 
	where $i$ is the imaginary unit and $h_{\gamma}$ is the unique element of $\M$ such that $\langle h, h_{\gamma} \rangle_{\M} = {_\Omega} \langle h, \gamma\rangle_{\Omega^*}$ for all $h\in \M$, see \cite[Lemma~8.2.2]{stroock_APV_V3}.\ The space $\M$ is usually called the Cameron--Martin space associated to $\P$ and the triple $(\Omega,\M,\P)$ is called an abstract Wiener space. 

\begin{example} \label{ex: wiener space}\
\begin{enumerate}
	\item[\textup{(i)}]
	Consider the classical case where $\P_W$ is the Wiener measure, i.e., the law of a standard Brownian motion.\ Then, the Cameron--Martin space is given by 
	\[
	\M_W := \bigg\{ h = \int_0^\cdot \dot{h} (s) \, ds : \dot{h} \in L^2 ([0, 1]; \bR^d) \bigg\},
	\] 
    see \cite[Section~8.1.2]{stroock_APV_V3}.\
	This example resembles the case of a time- and space-homogeneous Gaussian noise (sometimes referred to as {\em white noise}).
	\item[\textup{(ii)}]
	The framework also includes fractional settings, allowing for non-Markovian memory effects. For example, if $\P_H$ is the law of a fractional Brownian motion with Hurst parameter $H \in (0, 1)$, then $(\Omega, \M_H, \P_H)$ is an abstract Wiener space with Cameron--Martin space
	\[
	\M_H = \bigg\{ t \mapsto \int_0^t K_H (t, s) \dot{h} (s) \, ds : \dot{h} \in L^2 ([0, 1]; \bR^d)\bigg\}, 
	\]
	where, for $s,t\in [0,1]$, 
	\begin{align}  \label{eq: kernel K^H}
		K_H (t, s) := c_H \Bigg[ \bigg( \frac{t}{s} \bigg)^{H - \tfrac{1}{2}} (t - s)^{H - \frac{1}{2}} - (H - \tfrac{1}{2}) \, s^{\frac{1}{2} - H} \, \int_s^t u^{H - \frac{3}{2}} (u - s)^{H - \frac{1}{2}} \, du \Bigg]
	\end{align} 
	if $0< s<t\leq 1$, and $K_H(t,s)=0$ otherwise, see \cite[Lemma~3]{Budhiraja2025} and \cite[Equation (2.2)]{nua_06}.
\end{enumerate}
\end{example} 

In view of \cite[Theorem~8.2.6]{stroock_APV_V3}, there exists a unique linear isometric map $\mathcal{I} \colon \M \to L^2 (\Omega,\P)$ such that $\mathcal{I} (h_{\gamma}) = {_\Omega} \langle \, \cdot \, , \gamma \rangle_{\Omega^*}$ for all $\gamma \in \Omega^*$ and $\{ \mathcal{I} (h) : h \in \M \}$ is a Gaussian family in $L^2 (\Omega,\P)$. 

Let $\{ \pi_t : t \in [0, 1] \}$ be a continuous strictly monotone resolution of identity on $\M$, meaning it has the following properties:
\begin{enumerate}
	\item for every \(t \in [0, 1]\), \(\pi_t\colon \M\to \M\) is an orthogonal projection,
	\item $\pi_0 = 0, \pi_1 = \on{id}$,
	\item $\pi_{s} \M \subsetneq \pi_{t} \M$ for all $0 \leq s < t \leq 1$,
	\item $\lim_{s \to t} \pi_s h = \pi_t h$ for all $t \in [0, 1]$ and $h \in \M$.
\end{enumerate}
Finally, let $\mathcal{N}$ be the set of $\P$-null sets in $\Omega$ and define the filtration
\[
\cF_t := \sigma (\X_s, s \in [0, t]) \vee \mathcal{N} \quad\text{for } t \in [0, 1].
\] 
Throughout the paper, we impose the following standing assumption.

\begin{SA} \label{SA: filr}
For all \(t \in [0, 1]\), 
\[
\cF_t =  \sigma \big( \mathcal{I}  (\pi_t h) : h \in \M \big) \vee \mathcal{N}.
\] 	
\end{SA}

\begin{remark}\
\begin{enumerate}
		\item[(i)] Standing Assumption \ref{SA: filr} depends on the choice of $\{ \pi_t : t \in [0, 1] \}$.\ More precisely, we ask for the existence of a family $\{ \pi_t : t \in [0, 1] \}$ with the above properties such that Standing Assumption~\ref{SA: filr} is satisfied. 
		
		\item[(ii)] In the following, we explain that Standing~Assumption~\ref{SA: filr} holds for a general Volterra framework whose Cameron--Martin space \(\M\) is given through
	\[
	\M \equiv \Big\{ K \dot{h}: \dot{h} \in L^2 ([0, 1]; \bR^d) \Big\} \quad \text{with}\quad \| h \|_{\M} \equiv \| \dot{h} \|_{L^2 ([0, 1]; \bR^d)}, 
	\]
	where $K \colon L^2 ([0, 1]; \bR^d) \to C ([0, 1]; \bR^d)$ is a compact injective linear operator of Volterra-type:
	\[
	K f (t) := \int_0^1 K (t, s) f (s) \, ds, \quad\text{for } t \in [0, 1],
	\]
	with a Borel kernel $K \colon [0, 1] \times [0, 1] \to \bR^{d \times d}$ satisfying $K(t,s)=0$ for $0\leq t<s\leq 1$. 
	
	We briefly comment on the compactness hypotheses.\ By \cite[Example~4.1, p.\ 157]{kato_84}, the operator $K$ is compact whenever the Volterra kernel $(t, s) \mapsto K (t, s)$ is continuous.\ The continuity assumption is clearly not necessary.\ For example, if the kernel $K$ is given by \eqref{eq: kernel K^H}, corresponding to a fractional Brownian motion, then the operator $K$ is also compact.\ This can be deduced from \cite[Lemma~3.1]{DU_99} together with a straightforward application of the Arzel\`a--Ascoli theorem.	
	
	We now consider the continuous strictly monotone resolution of identity $(\pi_t)_{t \in [0, 1]}$, given by 
	\[
	(\pi_t h) (r) \equiv K \big(\dot{h} \1_{[0, t]}\big) (r) = \int_0^{t \wedge r} K (r, s) \dot{h} (s) \, ds \quad \text{for }t, r \in [0, 1]\text{ and } h \in \M.
	\]
	Under the assumption that, for every $t \in [0, 1]$, the
	\begin{equation} \label{eq: non-deg} \begin{split}
	\text{$L^2([0, t]; \bR^d)$-closure of }&
	\on{span} \big\{ K (r, \,\cdot \,)^\top \, v  : r \in [0, t], \, v \in \bR^d \big\} \\&\hspace{3cm}\text{ coincides with $L^2 ([0, t]; \bR^d)$,}
\end{split}
\end{equation} 
  it can be proved that the Standing Assumption~\ref{SA: filr} is satisfied, similarly to \cite[Theorem~4.3]{DU_99}. 
	
	We conclude this remark with a short discussion of the non-degeneracy condition \eqref{eq: non-deg}.\ It is satisfied, for example, when $K = \on{diag} (K_1, \ldots, K_d)$,
		where $K_i f = \int_0^1 K_i ( \, \cdot \, , s) f (s) \, ds$ is an injective operator from $L^2 ([0, t])$ to $C ([0, t])$ for every \(t \in [0, 1]\) and $i=1,\ldots, d$.\ This follows from the fact that
        $$\big\{ K_i(r, \, \cdot \,) : r \in [0, t] \big\}^\perp = \{ 0 \}\quad\text{in }L^2 ([0, t])\quad \text{for }i=1,\ldots, d.$$ 
        To prove this, let $i\in \{1,\ldots, d\}$ and take an arbitrary $u\in\big\{ K_i ( r,\,\cdot \,)  : r \in [0, t] \big\}^\perp$. Then,
		for every $r \in [0, t]$,
        $$\big\langle u, K_i (r,\,\cdot \,) \big\rangle_{L^2 ([0, t])} = K_i \big( u\1_{[0, t]}\big) (r) = 0.$$ By injectivity of $K_i$, we obtain that $u = 0$ in $L^2 ([0, t])$, which shows that the closure of $\on{span} \{ K_i (r,\, \cdot \,) : r \in [0, t] \}$ is $L^2 ([0, t])$.
        
		In particular, the non-degeneracy condition \eqref{eq: non-deg} holds for the kernel \eqref{eq: kernel K^H} associated to the fractional Brownian motion, for which Standing Assumption~\ref{SA: filr} is well-known to be satisfied by \cite[Theorem~4.3]{DU_99}.
		
		Finally, it is interesting to notice that the non-degeneracy condition is not implied by pure injectivity of the Volterra operator $K$, more precisely, the injectivity from $L^2([0, 1]; \bR^d)$ to $C ([0, 1]; \bR^d)$ alone.\ To give an example, take $d = 1$ and $K (t, s) = \1_{[0, t^2]} (s)$.\ Then, the operator $K$ is compact and injective from $L^2 ([0, 1])$ to $C ([0, 1])$, but the non-degeneracy condition fails, as, for all $0\leq r \leq t < 1$, $K (r, \, \cdot \,) \, v = 0$ on $(t^2, t] \not = \emptyset$.
        \end{enumerate}
\end{remark}

In the sequel, we investigate the question how additive noise from the abstract Wiener space $(\Omega, \M, \P)$ affects the deterministic controlled system \eqref{eq: deterministic controlled equation}. More specifically, we investigate a Laplace principle for the controlled {\em stochastic} system 
\begin{align} \label{eq: controlled stochastic equation}
	d Y_t = \mu (t, Y, \lambda_t) \, dt + \sqrt{\varepsilon} \, d \X_t, \quad Y_0 = x_0, 
\end{align}
as the noise parameter $\varepsilon > 0$ tends to zero.\ We start by describing this system in a mathematically rigorous way.\ 
Defining the set of all control strategies by
\begin{align*}
	\N &:= \Big\{ v \colon [0, 1] \to \Lambda, \, \text{Borel measurable} \Big\},
\end{align*}
we introduce the model \eqref{eq: controlled stochastic equation} through the associated solution map $\G \colon \N \times \Omega \to C ([0, 1]; \bR^d)$ to the following deterministic controlled equation: 
\begin{equation}\label{eq:Gamma-equation}
	\G_t (v, w)
	= x_0+\int_0^t\mu\big(s,\G (v,w),v_s\big)\,ds + w_t \quad \text{for } t \in [0, 1].
\end{equation}
From now on we endow $\N$ with the topology of convergence in measure w.r.t.\ the Lebesgue measure and identify all elements that are a.e.\ equal.\ This way, we turn $\N$ into a Polish space.
The following lemma can be established by standard Picard and Gronwall-type arguments.\ It records all facts about $\G$ that are used later. 

\begin{lemma}\label{lem: sol map}
	There exists a map $\G \colon \N \times \Omega \to C ([0, 1]; \bR^d)$ with the following properties:
	\begin{enumerate} 
	\item[\textup{(a)}] For every $v \in \N$ and $w \in \Omega$, the map \(t \mapsto \G_t (v, w)\) is the unique continuous solution to the equation~\eqref{eq:Gamma-equation}.
    \item[\textup{(b)}] \(\G\) is non-anticipative in the sense that 
	\[
	v, v' \in \N \text{ with } v = v' \text{ a.e.\ on } [0, t], \, w, w' \in \Omega \text{ with } w = w' \text{ on } [0, t] \implies \G_t (v, w) = \G_t (v', w').
	\] 
	\item[\textup{(c)}] For every $w \in \Omega$, the map $v \mapsto \G (v, w)$ is continuous from $\N$ to $C ([0, 1]; \bR^d)$, and $\G$ is equi-Lipschitz continuous in the second argument.\ More precisely, for all $v \in \N$, $w, w' \in \Omega$, and $t\in [0,1]$,
	\begin{align} \label{eq: Gronwall bound}
	\| \G (v, w) - \G(v, w') \|_t \leq e^L \| w - w' \|_t,
	\end{align} 
	where $L$ is the constant from Standing Assumption~\ref{SA: main coefficient}.
	\end{enumerate} 
\end{lemma}

\begin{proof}
    The existence of a unique solution map $\G$ is well-known; see, e.g., \cite[Theorem~14.30]{jacod79}. The non-anticipation property in (b) follows directly from the uniqueness, which holds also on every restricted time interval $[0, t]$.\ It remains to show the continuity properties from (c).\ 
    Using the Lipschitz condition from Standing Assumption~\ref{SA: main coefficient}~(b), for every $t \in [0, 1]$, we obtain that 
\begin{align*}
    \| \G (v, w) &- \G (v', w) \|_t 
    \\&\leq \int_0^t \big\| \mu \big(s, \G (v, w), v_s\big) - \mu \big(s, \G (v', w), v'_s\big) \big\| \, ds
    \\&\leq \int_0^t L \, \| \G (v, w) - \G (v', w) \|_s \, ds + \int_0^t \big\| \mu \big(s, \G (v', w), v_s\big) - \mu \big(s, \G (v', w), v_s'\big) \big\| \, ds. 
\end{align*}
Now, Gronwall's lemma yields that 
\begin{align*}
    \| \G (v, w) - \G (v', w) \|_\infty \leq e^L\, \int_0^1 \big\| \mu \big(s, \G (v', w), v_s\big) - \mu \big(s, \G (v', w), v_s'\big) \big\| \, ds. 
\end{align*}
Letting \(v \to v'\) in \(\N\) and using the dominated convergence theorem, which is applicable by the linear growth condition from Standing Assumption~\ref{SA: main coefficient}~(b), and the continuity assumption of \(\lambda \mapsto \mu (t, w, \lambda)\) from Standing Assumption~\ref{SA: main coefficient}~(c), we obtain that \(v \mapsto \G (v, w)\) is continuous from \(\N\) to~\(C ([0, 1]; \bR^d)\). 
Finally, as above, 
\begin{align*}
    \| \G (v, w) - \G (v, w') \|_t \leq \int_0^t L \, \| \G (v, w) - \G (v, w') \|_s \, ds + \| w - w' \|_t  \quad \text{for all } t \in [0, 1],
\end{align*}
so that \eqref{eq: Gronwall bound} follows from Gronwall's lemma. 
\end{proof}

Notice that (c) implies that \((v, w) \mapsto \G (v, w)\) is jointly continuous from \(\N \times \Omega\) to \(C ([0, 1]; \bR^d)\). In particular, the map is Borel measurable. 

\smallskip 

Returning to $(\Omega, \cF, \P)$, let $\A$ be the set of all $\Lambda$-valued $(\cF_t)_{t \in [0, 1]}$-progressively measurable processes.\ For every $a \in \A$, the process $t \mapsto Y_t := \G_t (a, \sqrt{\varepsilon} \, \X)$ satisfies the dynamics \eqref{eq: controlled stochastic equation}.
For suitable cost functionals $\varphi \colon C ([0, 1]; \bR^d) \to \bR$, we are interested in identifying the limit of
\begin{align*}
	V^\varepsilon := \varepsilon \inf_{a \in \A} \log \E \bigg[ \exp \bigg( \frac{\varphi (\G (a, \sqrt{\varepsilon} \X))}{\varepsilon} \bigg) \bigg]
\end{align*}
as $\varepsilon \to 0$.\ 
In Theorem~\ref{theo: main zusammengesetzt} below, we show that the limit can be described via a two-player zero-sum deterministic game, which we introduce in the following.\
Let $\mathscr{A}$ be the set of all measurable functions $\alpha \colon \N \to \M$ with the following non-anticipation property: for all $t \in [0, 1]$, 
\begin{align*}
	v, v' \in \N, \ v = v' \text{ a.e. on } [0, t] \implies  \pi_t \alpha (v) = \pi_t \alpha (v'). 
\end{align*}
Let $\mathscr{B}$ be the set of all measurable functions $\beta \colon \M \to \N$ with the non-anticipation property that, for all $t \in [0, 1]$, 
\begin{align*}
\pi_t u = \pi_t u' \implies \beta [u] = \beta [u'] \text{ a.e. on } [0, t].
\end{align*} 
The elements of $\mathscr{A}$ are typically called the strategies for the maximizing player and the elements of $\mathscr{B}$ are typically called the strategies for the minimizing player. 
For the cost 
\begin{align} \label{eq: game}
	C (v, u) := \varphi \big(\G (v, u)\big) - \frac{1}{2} \| u \|_{\M}^2, \quad \text{for }u \in \M \text{ and } v \in \N,
\end{align} 
the lower value of the game is given by 
\[
\lv := \sup_{\alpha \in \mathscr{A}} \inf_{v \in \N} C \big(v, \alpha [v]\big),
\] 
and the upper value is given by 
\[
\uv := \inf_{\beta \in \mathscr{B}} \sup_{u \in \M} C \big(\beta [ u ], u\big).
\] 
    In general, the upper and lower values need not satisfy any special relationship.\ In Theorem~\ref{theo: main zusammengesetzt}, we give conditions for \(\lv \leq \uv\) and in Theorem~\ref{prop.suff.fairness} we give conditions for \(\uv \leq \lv\), see also Discussion~\ref{dis: PDE approach} for comments on viscosity methods.\ Adapting standard terminology, we call the game with cost \(C\) {\em fair} if the upper and lower values coincide, i.e.,
\(
\lv = \uv.
\)
Our main result is the following. Its proof is given in Section~\ref{sec:proof} below.
\begin{theorem} \label{theo: main zusammengesetzt}
Suppose that \(\varphi \colon C ([0, 1]; \bR^d) \to \bR\) is bounded and uniformly continuous.\ Then, 
\begin{equation}\label{eq.main zusammen}
		\lv \leq \liminf_{\varepsilon \downarrow 0} V^\varepsilon \leq \limsup_{\varepsilon \downarrow 0} V^\varepsilon \leq \uv.
\end{equation}
In particular, if the game with cost \(C\) is fair, then 
\begin{align*}
    \lim_{\varepsilon \downarrow 0} V^\varepsilon = \lv = \uv.
\end{align*}
\end{theorem} 

In the next section we discuss the fairness condition from Theorem~\ref{theo: main zusammengesetzt} and provide worked examples. 

\section{Fairness and the Isaacs structure}
\label{subsec:fairness-isaacs}

Our first result relates fairness of the dynamic limiting game to the corresponding property of its static counterpart. The key tool for establishing this relation is Fan's \cite{KyFan1953} celebrated minimax theorem and the corresponding convexity and concavity notions, which are now recalled.\ We say that the cost $C$ is \textit{convex in the first variable} if, for all $v_1,v_2\in \N$ and all $\lambda\in (0,1)$, there exists $v_0\in \N$ such that, for all $u\in \M$,
\[
C(v_0,u)\leq \lambda C(v_1,u)+(1-\lambda)C(v_2,u)
\]
or, equivalently,
\[
 \varphi\big(\G (v_0, u)\big) \leq \lambda \varphi\big(\G (v_1, u)\big) +(1-\lambda)\varphi\big(\G (v_2, u)\big). 
\]
Analogously, we say that $C$ is \textit{concave in the second variable} if, for all $u_1,u_2\in \M$ and all $\lambda\in (0,1)$, there exists $u_0\in \M$ with $\|u_0\|_\M\leq \max\big\{\|u_1\|_\M,\|u_2\|_\M\big\}$ such that
\[
\lambda C(v,u_1)+(1-\lambda)C(v,u_2)\leq C( v,u_0)\quad \text{for all }v\in \N.
\]
Using \eqref{eq.main zusammen}, we have the following sufficient condition for fairness. 

\begin{theorem}\label{prop.suff.fairness}
Suppose that \(\varphi \colon C ([0, 1]; \bR^d) \to \bR\) is bounded and uniformly continuous, and assume that the cost $C$ is convex in the first variable and concave in the second variable.\ Then, the game is fair, and we have
\begin{equation}\label{eq.prop.minimax.fair}
\lim_{\varepsilon \downarrow 0} V^\varepsilon = \lv = \uv=\inf_{v\in\N}\max_{u\in\M} C(v,u)
  =
  \max_{u\in\M}\inf_{v\in\N} C(v,u).
\end{equation}
\end{theorem}

\begin{proof}
Since $\varphi$ is bounded, for all $\delta>0$, we have
\[
  - \| \varphi\|_\infty \leq \sup_{u\in \M}C(v,u)\le \|\varphi\|_\infty-\frac12\|u_\delta[v]\|_{\mathcal H}^2+\delta\quad \text{for all } v\in \N,
\]
where $u_\delta[v]\in \M$ is a $\delta$-optimizer for $C(v,u)$.\ Letting $\delta\downarrow 0$, we thus find that
\[
  \sup_{u\in\M} C(v,u)
  =
  \sup_{u\in B_R} C(v,u)
  \quad\text{for all } v\in\N,
\]
where
\[
  B_R:=\big\{u\in\M:\|u\|_{\M}\le R\big\}\quad \text{with}\quad  R>4\sqrt{\|\varphi\|_\infty}.
\]
As, by \cite[Corollary 8.3.10, p.\ 324]{zbMATH05851950}, the identity $\M\hookrightarrow \Omega$ is compact, weak convergence of a sequence $(u_n)_{n\in \mathbb N}\subseteq \M$ implies convergence of $(u_n)_{n\in \mathbb N}$ in $\Omega$.\ Together with the continuity of the solution map $\G$ in the second variable
and the continuity of $\varphi$, this yields weak sequential continuity
of the map
\begin{equation}\label{eq.map.minimax}
  \M\to \mathbb R,\quad u\mapsto \varphi\big(\G(v,u)\big)
\end{equation}
for all $v\in \N$. Since weak sequential continuity implies weak continuity on norm-bounded subsets of the separable Hilbert space $\M$, it follows that the map \eqref{eq.map.minimax} is weakly continuous on $B_R$ for all $v\in \N$.\ Moreover, the norm $\|\cdot\|_\M$ is weakly lower semicontinuous, so that the map $C(v,\,\cdot \,)$ is weakly upper semicontinuous on the weakly compact set $B_R$ for every $v\in \N$.\ Therefore, by Fan's minimax theorem \cite[Theorem~2]{KyFan1953},
\[
  \inf_{v\in\N}\max_{u\in B_R} C(v,u)
  =
  \max_{u\in B_R}\inf_{v\in\N} C(v,u).
\]
Consequently,
\[
\inf_{v\in\N}\sup_{u\in \M} C(v,u)=\inf_{v\in\N}\max_{u\in B_R} C(v,u)= \max_{u\in B_R}\inf_{v\in\N} C(v,u)\leq \sup_{u\in \M}\inf_{v\in\N} C(v,u)\leq \inf_{v\in\N}\sup_{u\in \M} C(v,u),
\]
i.e.,
\begin{equation}\label{eq.minimax}
  \inf_{v\in\N}\max_{u\in\M} C(v,u)
  =
  \max_{u\in\M}\inf_{v\in\N} C(v,u).
\end{equation}
Since constant strategies are admissible for both players, by \eqref{eq.main zusammen}, we have
\[
  \max_{u\in\M}\inf_{v\in\N} C(v,u)
  \le \lv
  \le \uv
  \le
  \inf_{v\in\N}\max_{u\in\M} C(v,u).
\]
Thus, the minimax identity \eqref{eq.minimax} implies
\[
  \lv=\uv=\inf_{v\in\N}\max_{u\in\M} C(v,u)
  =
  \max_{u\in\M}\inf_{v\in\N} C(v,u),
\]
 so that the game is fair and \eqref{eq.prop.minimax.fair} follows from Theorem \ref{theo: main zusammengesetzt}.
\end{proof}

For intuition, Theorem~\ref{prop.suff.fairness} shows that, under suitable convexity assumptions, fairness of a static game propagates to its dynamic counterpart.\ It is natural to investigate how sharp this condition is in general. As the next example shows, fairness may hold for the dynamic game while it fails for its static counterpart. 

\begin{example}
Let $d=1$, let \(\P = \P_W\) the standard Wiener measure, let $\Lambda=[-\kappa,\kappa]$ for \(\kappa > 0\), and set
\[
  \mu(t,\omega,\lambda)=\lambda.
\]
Fix $T\in(0,1)$ and let
\[
  \varphi(\omega) := - \min \left \{ \frac{\omega (T)^2}{2}, M \right \}, \quad \text{for }\omega \in C ([0, 1]).
\]
    Take \(x_0=0\)
and define
\[
  C_{\on{stat}}(\lambda ,x)
  :=
  -\min\left\{\frac{(\lambda T+x)^2}{2},M\right\}
  -\frac{x^2}{2T},
  \qquad
  \lambda \in[-\kappa,\kappa],\ x\in\mathbb R.
\]
The static outer values are
\begin{align}
  \lv_{\mathrm{stat}}
  &:=
  \sup_{x\in\mathbb R}
  \inf_{\lambda \in[-\kappa,\kappa]} C_{\on{stat}}(\lambda ,x)
  =
  -\min\left\{
    \frac{\kappa^2T^2}{2}, M
  \right\},
  \label{eq:negative-quadratic-static-lower}\\
  \uv_{\mathrm{stat}}
  &:=
  \inf_{\lambda \in[-\kappa,\kappa]}
  \sup_{x\in\mathbb R} C_{\on{stat}}(\lambda,x)
  =
  -\min\left\{
\frac{\kappa^2T^2}{2(1+T)}, M
  \right\}.
  \label{eq:negative-quadratic-static-upper}
\end{align}
Consequently,
\[
  \lv_{\mathrm{stat}}< \uv_{\mathrm{stat}}
  \quad\text{whenever}\quad
  M>
  \frac{\kappa^2T^2}{2(1+T)}.
\]
 On the other hand, as explained in Discussion~\ref{dis: PDE approach}, viscosity methods yield that the dynamic game is fair with value 
 \[
 I^\pm = I^+_{\on{stat}}.
 \]
Thus, in this example,
\[
  \lv_{\mathrm{stat}}
  <
  \uv_{\mathrm{stat}}
  \qquad\text{but}\qquad
  \lv = \uv.
\]
\end{example}

Next, we also discuss the viscosity method to establish fairness.

\begin{discussion} \label{dis: PDE approach}
(i) An important method to establish fairness of a two-player game is related to the Isaacs condition and viscosity theory for Isaacs PDEs.\ We sketch the basic idea without going into too much detail.\ For precise statements, we refer to the literature cited below.\ Consider first the
finite-dimensional Markovian specialization
\[
  \mu(t,\omega,\lambda)=b(t,\omega(t),\lambda),
  \qquad
  \varphi(\omega)= \psi\big(\omega(1)\big), 
\]
with \(\P = \P_W\), the standard Wiener measure. 
For $p\in\mathbb R^d$, the lower and upper Hamiltonians are
\begin{align*}
  H^-(t,x,p)
  &:=\sup_{z\in\mathbb R^d}\inf_{\lambda\in\Lambda}
    \left\{
       \langle p, b(t,x,\lambda) \rangle + \langle p, z \rangle -\frac12\|z\|^2
    \right\},\\
  H^+(t,x,p)
  &:=\inf_{\lambda\in\Lambda}\sup_{z\in\mathbb R^d}
    \left\{
      \langle p, b(t,x,\lambda) \rangle + \langle p, z\rangle - \frac12 \|z\|^2
    \right\}.
\end{align*}
Under suitable assumptions on $b$ and $\psi$, dynamic programming techniques show that the lower and upper values $I^\pm$ of the game with cost $C$ are time-zero values of viscosity solutions of the backward Isaacs PDEs 
\[
- \frac{u^{\pm}}{dt} (t, x) - H^{\pm} \big(t, x, \nabla_x u^\pm (t, x)\big) = 0, \quad u (1, x) = \psi (x).
\]
Due to the additive noise structure, the two controls occur in separate terms, which shows that 
\begin{equation*}
   H^-(t,x,p)
  =\inf_{\lambda\in\Lambda} \langle p, b(t,x,\lambda) \rangle
    +\frac12\|p\|^2
  = H^+(t,x,p).
  \label{eq:brownian-isaacs-identity}
\end{equation*}
In other words, the Isaacs condition is satisfied and the Isaacs PDEs for $I^\pm$ are the same. As a consequence, the game is fair whenever the PDE has a unique viscosity solution. For a textbook treatment of this approach, we refer to \cite[Chapter~XI]{FleSon_06}.

A similar strategy also works in settings with a fully path-dependent coefficient $\mu$, with the important exception that then viscosity theory for path-dependent PDEs is required, see \cite{peng_et_all_PPDE_survey}. We refer to the recent works \cite{PTZ_20, TZ_26} for results in this direction. 

\smallskip
(ii) We emphasize that the above discussion also extends to certain fractional frameworks whenever we restrict our attention to terminal costs. To simplify the exposition, let $d=1$, $\Lambda=[-\kappa,\kappa]$, and set
\[
  \mu(t,\omega,\lambda)=\lambda \quad \text{for } t \in [0, 1],\, \omega \in C ([0, 1] ), \text{ and } \lambda \in \Lambda.
\]
Fix $T\in(0,1)$ and let
\[
  \varphi(\omega)=\psi\big(\omega(T)\big) \quad \text{for } \omega \in C ([0, 1]), 
\]
where $\psi \colon\mathbb R\to\mathbb R$ is bounded and uniformly continuous.
Lastly, take $K=K_H$ to be the fractional kernel from \eqref{eq: kernel K^H} with Hurst parameter $H\in(0,1)$. For $h \in \M$ write $h=K_H \dot{h}$ with $\dot{h}\in L^2([0,1])$.  Under this identification, for all $h \in \M$ and $t \in [0, 1]$, 
\[
  \|h\|_\M^2=\int_0^1 |\dot{h} (s)|^2\,ds,
  \qquad
  \pi_t h =K_H(\dot{h}\mathbf 1_{[0,t]}).
\]
With $k_T :=K_H(T, \, \cdot \,)\1_{[0,T]}$, the values of the game with cost $C$ agree with those of the one-dimensional auxiliary differential game
\begin{equation*}
  d Y_t = \big( \lambda_t + k_T(t) \, \dot{h} (t) \big) \, dt,
  \qquad Y_0 = x_0,
  \label{eq:fractional-terminal-lift}
\end{equation*}
with cost
\begin{equation*}
  J(y, h)
  :=\psi(y_T)-\frac12\int_0^T | \dot{h} (t) |^2\,dt.
  \label{eq:fractional-terminal-payoff}
\end{equation*}
In other words, we transformed the game associated to fractional noise into a game associated to Brownian noise with a time-dependent coefficient. The latter game is covered by the viscosity approach outlined in part (i) of this discussion, see again \cite{FleSon_06}.
\end{discussion} 

We conclude this section with an explicitly solved example that explains the influence of memory effects of the noise on the system. 
\begin{example} \label{ex: main example}
	Let $K=K_H$ be the fractional kernel from \eqref{eq: kernel K^H} with Hurst parameter $H\in(0,1)$ in dimension $d=1$, and take a compact action set $\Lambda= [- \kappa, \kappa]$
	for some $\kappa>0$.  Moreover, let
	\[
	\mu(t,\omega,\lambda)=\lambda
	\quad \text{for }t\in[0,1],\, \omega\in C ([0, 1]),\text{ and } \lambda\in\Lambda .
	\]
	Then, for $t\in[0,1]$, $v\in\N$,  and $u\in\M$,
	\begin{align*} \label{eq: example controlled system}
	\G_t(v,u)=x_0+\int_0^t v_s\,ds+u_t .
	\end{align*} 
	We consider a cost $\varphi (\omega) := \psi \big(\omega (T)\big)$ for the bounded uniformly continuous function 
    \[
    \psi (x) = \min \{ x^2/2, M \}, \quad \text{with } M > \frac{[(|x_0|-\kappa T)_+]^2}{2(1-T^{2H})^2},
    \]
    and a terminal time horizon $T \in (0, 1)$.\ 
    It is straightforward to show that
	\begin{align*}
	\lv = \uv &= \frac{[(|x_0|-\kappa T)_+]^2}
  {2(1-T^{2H})}.
	\end{align*} 
	In particular, Theorem~\ref{theo: main zusammengesetzt} yields that 
	\[
	\lim_{\varepsilon \downarrow 0} V^\varepsilon = \frac{[(|x_0|-\kappa T)_+]^2}
  {2(1-T^{2H})}.
	\] 
	We observe that this function is decreasing in $H$.\
	On an intuitive level, this monotonicity may be interpreted as an indication that long-range dependence has more influence on the small-noise limit than short-range dependence.
\end{example}

\section{Proof of Theorem~\ref{theo: main zusammengesetzt}}\label{sec:proof}
This section is dedicated to the proof of our main result, Theorem~\ref{theo: main zusammengesetzt}.\ Namely, we prove the following inequalities:
\begin{align} \label{eq: main}
	\lv \leq \liminf_{\varepsilon \downarrow 0} V^\varepsilon \leq \limsup_{\varepsilon \downarrow 0} V^\varepsilon \leq \uv.
\end{align} 

\subsection{Functional analytic tools}\label{sec:funkana}
Before we turn to the proof of \eqref{eq: main}, we develop some mathematical tools which play a crucial role in the proof of \eqref{eq: main}.\ We start with the following version of the celebrated multiplicity theorem for normal operators, cf.\ \cite[Theorem 10.1, p.\ 293]{conway}.\ For the sake of a self-contained exposition, we provide a short proof.

For a family $\eta=(\eta_m)_{m\in M}$ of probability measures on $[0,1]$ with $M\subseteq \mathbb N$ nonempty, we define the Hilbert direct sum by
\[
 \L:= \bigoplus_{m\in M}L^2\big([0,1],\eta_m\big)
  :=
  \bigg\{
    f=(f_m)_{m\in M} \in \prod_{m\in M} L^2\big([0,1],\eta_m\big):
    \|f\|_\L<\infty
  \bigg\}
\]
with
\[
\|f\|_\L:=\bigg(\sum_{m\in M} \|f_m\|^2_{L^2([0,1],\eta_m)}\bigg)^{1/2}\quad \text{for }f=(f_m)_{m\in M}\in\prod_{m\in M} L^2\big([0,1],\eta_m\big).
\]
Then, $\L$ is a Hilbert space with the inner product
\[
  \langle f,g\rangle_\L
  :=
  \sum_{m\in M}
  \langle f_m,g_m\rangle_{L^2([0,1],\eta_m)}
  \quad\text{for }
  f=(f_m)_{m\in M},\, g=(g_m)_{m\in M}\in \L.
\]

\begin{proposition}\label{prop.spectral.mult}
  There exists a nonempty set $M\subseteq \mathbb N$, a family of atomless probability measures $\eta=(\eta_m)_{m\in M}$ on $[0,1]$, and a unitary operator
  \[
  U\colon \M\to \L,\quad h\mapsto (\dot h_m)_{m\in M}
  \]
  with inverse $K:=U^{-1}\colon \L\to \M$ such that
  \begin{equation}\label{eq.spectral.decomposition}
  \pi_t h=K\big( (\dot h_m\1_{[0, t]})_{m\in M}\big)\quad \text{for all }t\in [0,1]\text{ and }h\in \M.
  \end{equation}
\end{proposition}

\begin{proof}
We prove the result by decomposing the spectral measure associated with the spectral family $(\pi_t)_{t\in [0,1]}$ into cyclic subspaces.

First, the increasing family $(\pi_t)_{t\in[0,1]}$ induces a projection-valued (PV)
measure
\[
  E\colon\mathcal B([0,1])\to L(\mathcal H)
\]
defined on half-open intervals by
\[
  E\big((s,t]\big) := \pi_t-\pi_s,
  \qquad 0\le s<t\le1,
\]
and extended to the Borel sets with \(E \big( \{ 0 \} \big) := 0\).\ In particular,
\[
  E\big([0,t]\big)=\pi_t.
\]
Since the projections are increasing, \(\pi_s\pi_t=\pi_t\pi_s=\pi_s\) for
\(0\leq s\le t\leq 1\), and therefore \(\pi_t-\pi_s\) is again an orthogonal projection.\ Moreover, the strong continuity of \([0,1]\to L(\M),\, t\mapsto\pi_t\) implies that \(E\) has no atoms, i.e., $E(\{t\})= 0$ for all $t\in [0,1]$.\ Indeed, 
\[
  E(\{t\})h
  =
  \pi_th-\lim_{s\uparrow t}\pi_sh
  =
  0\quad \text{for }t\in(0,1]\text{ and }h\in \M,
\]
and, by definition, $E(\{0\})=\pi_0=0$.

Now fix a vector \(x\in\mathcal H\) with $\|x\|_\M=1$, and define its scalar spectral
measure by
\[
  \eta_x(B):=\langle E(B)x,x\rangle_\M
  \quad \text{for } B\in\mathcal B([0,1]).
\]
Since $\|x\|_{\mathcal H}=1$ and $\pi_1=\on{id}$, it follows that 
\[
  \eta_x([0,1])
  =
  \langle E([0,1])x,x\rangle_{\mathcal H}
  =
  \|x\|_{\mathcal H}^2
  =
  1,
\]
so that $\eta_x$ is an atomless probability measure.\ Let
\[
  \mathcal H_x
  :=
  \overline{\operatorname{span}}
  \big\{E(B)x:B\in\mathcal B([0,1])\big\}
\]
be the cyclic subspace generated by $x$.\ We claim that $\M_x$ is unitarily equivalent to $L^2\big([0,1],\eta_x\big)$.\ For a simple function
\[
  f=\sum_{i=1}^n a_i\1_{B_i}
\]
with $a_1,\ldots, a_n\in \mathbb R$, $B_1,\ldots, B_n\in \mathcal B([0,1])$, and $n\in \mathbb N$, define
\[
  K_x f
  :=
  \sum_{i=1}^n a_iE(B_i)x.
\]
Using the fact that $E$ is a PV measure, we obtain
\[
\begin{aligned}
  \|K_x f\|_{\mathcal H}^2
  &=
  \sum_{i,j=1}^n a_ia_j\big\langle
     E(B_i)x, E(B_j)x
  \big\rangle_\M = \sum_{i,j=1}^n a_ia_j\big\langle
     E(B_i)E(B_j)x, x
  \big\rangle_\M \\
  &=\sum_{i,j=1}^n a_ia_j\big\langle
     E(B_i\cap B_j)x, x
  \big\rangle_\M=
  \int_{[0,1]} |f(\theta)|^2\,\eta_x(d\theta).
\end{aligned}
\]
Therefore $K_x$ extends uniquely to an isometry 
\[
K_x\colon L^2\big([0,1],\eta_x\big)\to \M.
\]
By construction, its range is precisely $\M_x$, so that $K_x$ is unitary.\ Moreover, for all $f\in L^2\big([0,1],\eta_x\big)$ and  $B\in \mathcal B([0,1])$,
\begin{equation}\label{eq.property.cyclic}
  E(B)K_x f
  =
  K_x(\1_B f).
\end{equation}
Thus, on the cyclic subspace $\M_x$, the PV measure $E$ is represented by multiplication by indicator functions, and it remains to decompose $\M$ into mutually orthogonal cyclic
subspaces.\ Using Zorn's lemma, there exists a maximal family of non-zero mutually orthogonal cyclic
subspaces $(\M_{x_m})_{m\in M}$ with normalized cyclic generators $(x_m)_{m\in M}\subseteq \M$.\ Since $E\big([0,1]\big)=\pi_1=\on{id}$, it follows that the family $(x_m)_{m\in M}$ is orthonormal.\ Since $\M$ is separable, it follows that $M$ is at most countable.\ Thus we may take $M\subseteq\mathbb N$. 

Finally, we show that the cyclic subspaces span all of $\M$.\ To that end, let $y\in \M\setminus\{0\}$ with 
$$y\perp \overline{\rm span}\bigg(\bigcup_{m\in M}\M_{x_m}\bigg)$$
Then, for all Borel sets $A,B\in \mathcal B([0,1])$ and $m\in M$,
\[
  \langle E(A)y,E(B)x_m\rangle_\M  =
  \langle y,E(A)E(B)x_m\rangle_\M  =
  \langle y,E(A\cap B)x_m\rangle_\M  =0,
\]
so that $\M_y\perp \M_{x_m}$ for all $m\in M$, which contradicts the maximality of the family $(\M_{x_m})_{m\in M}$.\ 

For each $m\in M$, let 
\[
  P_m\colon\M\to\M_{x_m}
\]
denote the orthogonal projection onto the $m$-th cyclic subspace $\M_{x_m}$.\ Define
\[
  U\colon \M\to
  \L
\]
by
\[
  Uh
  := \big(\dot h_m\big)_{m\in M}:=
  \bigl(K_m^{-1}P_mh\bigr)_{m\in M}\quad \text{for }h\in \M.
\]
By definition, its inverse $K:=U^{-1}$ is given by
\[
  K\big((f_m)_{m\in M}\big)
  =
  \sum_{m\in M}K_m f_m\quad \text{for }f=(f_m)_{m\in M}\in \L,
\]
where the right-hand side converges in $\M$.\ By \eqref{eq.property.cyclic} with $B=\1_{[0,t]}$ for $t\in [0,1]$, it follows that 
\[
 \pi_t h=\sum_{m\in M} \pi_tK_m \dot h_m =\sum_{m\in M} K_m \big(\1_{[0,t]}\dot h_m\big)=K\big( (\dot h_m\1_{[0, t]})_{m\in M}\big)\quad \text{for all }h\in \M.
\]
The proof is complete.
\end{proof}

We point out that, in the cases described in Example~\ref{ex: wiener space}, the spectral multiplicity, i.e., the cardinality of the set $M$, equals the dimension $d$.

The following technical auxiliary result plays a central role in the subsequent discussion.

\begin{lemma} \label{SA: partition} 
	For every \(n \in \mathbb{N}\), there exists a partition \(0 = t^n_1 < \dots < t^n_{N_n} = 1\) with $N_n\in \mathbb N$
	and vectors \(e^n_1, \dots, e^n_{N_n}, u^n_1, \dots, u^n_{N_n} \in \M\) with the following properties: 
	\begin{enumerate}
		\item[(i)] For every \(j \in \{1,\ldots,N_n\}\), it holds \(e^n_j \in \pi_{t^n_j} \M\) and \(\pi_t u^n_j = 0\) for all \(0\leq t \leq t^n_j\). 
		\item[(ii)] The operator \(Q_n \colon \M\to \M\), defined by 
		\[
		Q_n h := \sum_{j = 1}^{N_n} \langle h, e^n_j \rangle_\M\, u^n_j \quad \text{for all }h \in \M, 
		\]  
		is a contraction and has the property that, for every \(R > 0\), 
		\[
		\sup_{ \substack{h \, \in \, \M \\ \| h \|_{\M} \leq R }} \| Q_n h - h \|_\infty \to 0 \quad\text{as } n \to \infty. 
		\] 
	\end{enumerate}
\end{lemma}

\begin{proof}
Let $\eta=(\eta_m)_{m\in M}$ with $M\subseteq\mathbb N$ be the family of atomless probability measures constructed in Proposition \ref{prop.spectral.mult}.\ Then, for every $m\in M$ and every
$n\in\mathbb N$, we may choose a partition
\[
  0=t_{m,0}^n<t_{m,1}^n<\cdots<t_{m,n}^n=1
\]
with
\[
  \eta_m\bigl((t_{m,k-1}^n,t_{m,k}^n]\bigr)=\frac1n
  \quad \text{for all } k=1,\ldots,n.
\]
%If \(m\notin M\), we ignore that component.
For $m\in M$, $n\in \mathbb N$, and $k=1,\ldots,n$, set
\[
  b_{m,k}^n
  :=
  \sqrt n\,\1_{(t_{m,k-1}^n,t_{m,k}^n]}
  \in L^2\big([0,1],\eta_m\big).
\]
Then,
\[
  \|b_{m,k}^n\|_{L^2([0, 1], \eta_m)}^2
  =
  n\,\eta_m\big((t_{m,k-1}^n,t_{m,k}^n]\big)
  =
  1\quad \text{for all }m\in M,\, n\in \mathbb N,\text{ and } k=1,\ldots,n,
\]
and therefore $b_{m,1}^n,\ldots,b_{m,n}^n$
is an orthonormal family in $L^2\big([0,1],\eta_m\big)$ for all $m\in M$.\

For each $m\in M$ and $n\in \mathbb N$, let
\[
  P_m^n\colon L^2\big([0,1],\eta_m\big)\to L^2\big([0,1],\eta_m\big),\quad f\mapsto \sum_{k=1}^{n-1}
  \langle f,b_{m,k}^n\rangle_{L^2([0,1],\eta_m)}\,b_{m,k+1}^n.
\]
Then, $P_m^n$ is a finite-rank contraction on $L^2\big([0,1],\eta_m\big)$ for all $m\in M$ and $n\in \mathbb N$.

Now, for $n\in \mathbb N$, we define the operator
\[
   Q_n\colon\M\to\M,\quad h\mapsto K\big((P_m^n\dot h_m^n)_{m\in M}\big)=\sum_{\substack{m\in M\\ m\leq n}} \sum_{k=1}^{n-1} \big\langle h,K_mb_{m,k}^n\big\rangle_\M K_m b_{m,k+1}^n,
\]
where, for $m\in M$, 
\[
  \dot h_m^n :=
  \begin{cases}
    \dot h_m, &  m\le n,\\
    0, & m>n.
  \end{cases}
\]
Since $K$ and $U$ are unitary and $P_m^n$  is a finite-rank contraction on $L^2\big([0,1],\eta_m\big)$ for all $m\in M$ and $n\in \mathbb N$, it follows that $Q_n$ is a finite-rank contraction on $\M$ for each $n\in \mathbb N$. 

Observe that, by \eqref{eq.property.cyclic},
\[
\pi_t K_m b^n_{m,k+1}=K_m \big(\1_{[0,t]}b^n_{m,k+1}\big)
\]
for all $t\in [0,1]$, $m\in M$, $n\in \mathbb N$, and $k=1,\ldots, n-1$, so that
\[
 \pi_t K_m b^n_{m,k+1}=0
\]
for all $0\leq t\leq t_{m,k}^n$, $m\in M$, $n\in \mathbb N$, and $k=1,\ldots, n-1$.\ Moreover, by construction $K_mb_{m,k}^n\in \pi_{t_{m,k}^n}$ for all $m\in M$, $n\in \mathbb N$, and $k=1,\ldots, n-1$.

Last but not least, we need to find one common partition for each $n\in \mathbb N$ that is independent of $m\in M$.\ For fixed $n\in \mathbb N$, only
the components $m\in M$ with $m\le n$ and finitely many $k=1,\ldots, n-1$ are used.\ Hence
we define a common partition by indexing all these points as $t_1^n,\ldots, t^n_{N_n}$ with $N_n\in \mathbb N$ such that
\[
  \{t_1^n,\ldots, t^n_{N_n}\}=
  \big\{t_{m,k}^n: m\in M \text{ with } m\le n\text{ and } k=1,\ldots,n-1\big\}\cup \{0,1\}.
\]

Next, we prove that $Q_n^*h\to h$ as $n\to \infty$ for all $h\in \M$.\ By definition of $(Q_n)_{n\in \mathbb N}$ and the norm on $\L$, it is sufficient to show that $P_m^{n,*}f\to f$ as $n\to \infty$ for all $m\in M$ and $f\in L^2([0,1],\eta_m)$, where
\[
 P_m^{n,*}\colon L^2\big([0,1],\eta_m\big)\to L^2\big([0,1],\eta_m\big),\quad f\mapsto \sum_{k=1}^{n-1}
  \langle f,b_{m,k+1}^n\rangle_{L^2([0,1],\eta_m)}\,b_{m,k}^n
\]
is the adjoint of $P_m^n$ for $m\in M$ and $n\in \mathbb N$.\ For $m\in M$ and $f=\1_{(a,b]}$ with $0\leq a< b\leq 1$,
\begin{align}
\notag \|P_m^{n,*} f-f\|_{L^2([0,1],\eta_m)}^2&=\eta_m\big((t_{m,n-1}^n,t_{m,n}^n]\cap (a,b]\big)\\
&\quad+\sum_{k=1}^{n-1} \int_{(t_{m,k-1}^n,t_{m,k}^n]} \Big|n\eta_m\big((t_{m,k}^n,t_{m,k+1}^n]\cap (a,b]\big)-\1_{(a,b]}(\theta)\Big|^2\, \eta_m(d\theta).\label{eq.estimate_L2norm}
\end{align}
We analyze the sum on the right-hand side in \eqref{eq.estimate_L2norm}.\ To that end, let $k_a,k_b\in \{1,\ldots, n\}$ with $$a\in (t_{m,k_a-1}^n,t_{m,k_a}^n]\quad\text{and}\quad b\in (t_{m,k_b-1}^n,t_{m,k_b}^n].$$
Then, $k_a\leq k_b$. Let $k\in \{1,\ldots, n-1\}$.\ If $k\leq k_a-2$ or $k\geq k_b+1$, it follows that 
\[
 (t_{m,k-1}^n,t_{m,k}^n]\cap (a,b]=\emptyset\quad \text{and}\quad (t_{m,k}^n,t_{m,k+1}^n]\cap (a,b]=\emptyset.
\]
On the other hand, if $k_a+1\leq k\leq k_b-2$, it follows that
\[
 (t_{m,k-1}^n,t_{m,k}^n]\cap (a,b]=(t_{m,k-1}^n,t_{m,k}^n]\quad \text{and}\quad (t_{m,k}^n,t_{m,k+1}^n]\cap (a,b]=(t_{m,k}^n,t_{m,k+1}^n].
\]
Therefore, the only summands which can be different from zero on the right-hand side in \eqref{eq.estimate_L2norm} are $k=k_a-1,k_a,k_b-1,k_b$, and $\eta_m\big((t_{m,n-1}^n,t_{m,n}^n]\cap (a,b]\big)$, which implies that
\[
\|P_m^{n,*} f-f\|_{L^2([0,1],\eta_m)}^2\leq \frac{5}n\to 0\quad \text{as }n\to \infty.
\]
Since the set $\big\{(a,b]: 0\leq a< b\leq 1\big\}$ is an intersection-stable generator of $\mathcal B((0,1])$ and $\eta_m\big([0,1]\big)=\eta_m\big((0,1]\big)$ due to the fact that $\eta_m$ is atomless for all $m\in M$, it follows that 
\[
{\rm span} \big\{\1_{(a,b]}: 0\leq a< b\leq 1\big\}
\]
is dense in $L^2\big([0,1],\eta_m\big)$ for all $m\in M$, see, e.g., \cite[Proposition A.1]{zbMATH08149100}.\ Since also $P_m^{n,*}$ is a linear contraction for all $m\in M$ and $n\in \mathbb N$, we thus obtain that $P_m^{n,*}f\to f$ as $n\to \infty$ for all $m\in M$ and $f\in L^2([0,1],\eta_m)$.\

Finally, by \cite[Corollary 8.3.10, p.\ 324]{zbMATH05851950}, the identity $J\colon \M\hookrightarrow \Omega$ is compact.\ By Schauder's theorem, this is equivalent to the compactness of $J^*\colon \Omega^*\hookrightarrow \M$, which implies that
\[
\|(Q_n^*-{\rm id}_\M)J^*\|_{L(\Omega^*,\M)}\to 0\quad \text{as }n\to \infty,
\]
and therefore $\|J(Q_n-{\rm id}_\M)\|_{L(\M,\Omega)}\to 0$ as $n\to \infty$.\ The proof is complete.
\end{proof}

Next, we prove an adapted projection approximation result, which relies in a crucial manner on Lemma~\ref{SA: partition}.

\begin{lemma}\label{lem:projection}
	There are Borel maps $P_n \colon \Omega \to \M$, $n\in\mathbb{N}$, with the following properties:
	\begin{enumerate}[label=\textup{(\roman*)}]
		\item For every $n\in \mathbb N$, $P_n$ is adapted in the sense that \(\pi_t P_n\) is \(\cF_t\)-measurable for every \(t \in [0, 1]\).
		\item For every \(n \in \mathbb{N}\), there exists a \(\P\)-null set \(N\in \cF\) such that, for all \(\eta \in \M\), \(P_n (\omega + \eta) = P_n (\omega) + Q_n(\eta)\) for all \(\omega \not \in N\), where \(Q_n \colon \M \to \M\) is a linear contraction such that, for every $R >0$,
		\begin{equation}\label{eq:projection-uniform}
			\sup_{\substack{u\in \M\\ \|u\|_{\M}\le R}}\|Q_n u -u\|_\infty\to0 \quad \text{as } n \to \infty.
		\end{equation}
		\item For every \(n \in \mathbb{N}\), \(\E [ \| P_n \|_{\mathcal{H}}^2 ] < \infty\). 
	\end{enumerate}
\end{lemma}

\begin{proof}
We begin with a preparatory step.\ Throughout, fix $n \in \mathbb{N}$, let \(e^n_1, \dots, e^n_{N_n}\) as in Lemma~\ref{SA: partition}, and set $E := \big\{ e^n_k : 1 \leq k \leq N_n \big\} \subset \M$. Recall from \cite[Lemma~8.2.2~(i)]{stroock_APV_V3} that $\{ h_{x^*} : x^* \in \Omega^* \}$ is dense in $\M$. Hence, for every \(e \in E\), there exists a sequence $(\omega^*_{m, e})_{m \in \mathbb{N}} \subset \Omega^*$ such that $h_{x^*_{m, e}} \to e$ as $m \to \infty$.\ Thanks to \cite[Theorem~8.2.6]{stroock_APV_V3}, for all $m \in \mathbb{N}$ and $e \in E$, $\mathcal{I} (h_{x^*_{m, e}}) = \, _{\Omega} \langle \, \cdot \, , x^*_{m, e} \rangle_{\Omega^*}$ and $h \mapsto \mathcal{I} (h)$ is continuous from $\M$ to $L^2 (\Omega, \P)$. Hence, after passing to a subsequence if necessary, we may assume that $\P$-a.s. $_{\Omega} \langle \, \cdot \, , x^*_{m, e} \rangle_{\Omega^*} = \mathcal{I} (h_{x^*_{m, e}}) \to \mathcal{I} (e)$.
For $e \in E$, set
\[
\Omega_e := \left \{ \omega \in \Omega :  \, _{\Omega} \langle \omega, x^*_{m, e} \rangle_{\Omega^*} \text{ converges as \(m \to \infty\)} \right \}. 
\]
By \cite[Lemma~1.11]{kallenberg}, the set \(\Omega_e\) is Borel and, as \(\P\)-a.s.\ $_{\Omega} \langle \, \cdot \, ,x^*_{m, e} \rangle_{\Omega^*} = \mathcal{I} (h_{x^*_{m, e}}) \to \mathcal{I} (e)$, also $\P$-full. Now, define $\mathcal{L} \colon E \to L^2 (\Omega, \P)$ by 
\[
\mathcal{L} (e) (\omega) := \begin{cases} \displaystyle\lim_{m \to \infty} \, _{\Omega} \langle \omega, x^*_{m, e} \rangle_{\Omega^*}, & \omega \in \Omega_e, \\ 0, & \omega \not \in \Omega_e, \end{cases} 
\]
where we notice that, for $e \in E$, $\P$-a.s. 
\[
\mathcal{L} (e) = \lim_{m \to \infty} \, _{\Omega} \langle \, \cdot \,, x^*_{m, e} \rangle_{\Omega^*} = \lim_{m \to \infty} \mathcal{I} (h_{x^*_{m, e}}) = \mathcal{I} (e).
\]
In other words, \(\mathcal{L} (e)\) is a $\P$-version of $\mathcal{I} (e)$. After this preparatory step, we are in the position to start the main proof of (i)-(iii).

\smallskip 
(i) We define 
\[
P_n (\omega) := \sum_{j = 1}^{N_n} \mathcal{L} (e^n_j) (\omega) \, u^n_j \quad \text{for }n\in \mathbb N \text{ and }\omega \in \Omega. 
\] 
Recalling that $\P$-a.s. $\mathcal{L} (e) = \mathcal{I} (e)$, we have $\P$-a.s.
\[
\pi_t P_n = \sum_{j = 1}^{N_n}  \mathcal{I} (e^n_j) \, \pi_tu^n_j, 
\]
where the right hand side is $\cF_t$-measurable by Standing Assumption~\ref{SA: filr} and Lemma~\ref{SA: partition}~(i).\ As $\mathcal{F}_t$ is $\P$-complete, this implies that \(\pi_t P_n\) is $\cF_t$-measurable as well, see \cite[Lemma~1.27]{kallenberg}.\ 

\smallskip 
(ii) Notice that, for every $\omega \in \Omega_e$ and $h \in \M$, 
\[
\lim_{m \to \infty} \, _{\Omega} \langle \omega + h, x^*_{m, e} \rangle_{\Omega^*} \text{ exists and equals } \mathcal{L} (e) (\omega) + \langle h, e \rangle_\M.
\] 
Consequently, $\omega + h \in \Omega_e$ and 
\[
\mathcal{L} (e) (\omega + h) = \mathcal{L} (e) (\omega) + \langle h, e \rangle_\M.
\] 
Setting $F := \bigcap_{e \in E} \Omega_e$, we get, for all $\omega \in F$ and $h \in \M$, 
\[
P_n (\omega + h) = P_n (\omega) + Q_n (h), 
\] 
where $Q_n$ is as in Lemma~\ref{SA: partition}. As $F$ is $\P$-full, this entails the claim of (ii). 

\smallskip
(iii) Part (iii) follows from the fact that, for every $e \in E$, $\P$-a.s. $\mathcal{L} (e) = \mathcal{I}(e)$, together with the fact that each $\mathcal{I} (e)$ is a Gaussian random variable.
\end{proof}

\subsection{Bou\'e--Dupuis formula}\label{sec:boue.dupuis}

For the proof of the inequalities in \eqref{eq: main}, we need the following version of the Bou\'e--Dupuis formula, which is a consequence of the main result in \cite{Z_09} that extends the formula from the standard Brownian motion to abstract Wiener spaces.\ Let $\mathbb{H}$ be the set of all $\M$-valued random variables $h$ on $(\Omega, \cF, \P)$ such that, for every $t \in [0, 1]$, $\pi_t h$ is $\cF_t$-measurable. 

\begin{proposition} \label{prop: BD formula}
	For all $\varepsilon > 0$, 
\begin{align*}
 V^\varepsilon &= \inf_{a \in \A} \sup_{h \in \mathbb{H}} \E \bigg[ \varphi \Big( \G (a (\sqrt{\varepsilon} \, \X + h), \sqrt{\varepsilon} \, \X + h) \Big) - \frac{1}{2} \| h \|^2_{\M}\bigg]
 \\
 &= \inf_{a \in \A} \sup_{h \in \mathbb{H}} \E \bigg[ \varphi \Big( \G (a (\X + h / \sqrt{\varepsilon} ),  \sqrt{\varepsilon} \, \X + h) \Big) - \frac{1}{2} \| h \|^2_{\M}\bigg].  
  %  \log \E \big[ e^{\psi (K \, *\, \X)} \big] = \sup_{h \in \mathbb{H} } \E \Big[ \psi (K * \X + h) - \frac{1}{2} \| h \|_{\M} \Big]. 
\end{align*}
\end{proposition} 
\begin{proof} 
	The first equation follows directly from Standing Assumption~\ref{SA: filr} and \cite[Theorem~3.2]{Z_09}.\ The former is needed in order to ensure that the control strategies are adapted to $(\cF_t)_{t \in [0, 1]}$.\ The second equation follows from a simple reparametrization, namely exchanging $a (\, \cdot \,)$ with $a (\, \cdot \, / \sqrt{\varepsilon})$, which does not affect the optimization.
\end{proof}

We are now in the position to present our proofs for \eqref{eq: main}, given in the following two subsections.

\subsection{Proof of the first inequality in \eqref{eq: main}}\label{sec:first.ineq}

We start with the following approximation result.\ To ease its formulation, we call a strategy $\alpha \in \mathscr{A}$ {\em simple}, if 
\begin{align*}
\pi_t \alpha [ v ] = \sum_{k = 0}^{N-1} g_k [v] \pi_t u^k \1_{\{t_k < t\}}, \quad\text{for all } v \in \N,
\end{align*} 
with $N\in \mathbb N$, \(0 = t_0 < t_1 < \dots < t_N = 1, u^k \in \M\), and 
\begin{align} \label{eq: nonanti gk}
v = v' \text{ a.e. on \([0, t_k]\)} \implies g_k [v] = g_k [v'].
\end{align}

\begin{lemma} \label{lem: piecewise constant approximation 1}
Let $\alpha\in \mathscr{A}$ be bounded, $\delta > 0$, and assume that the cost function $\varphi$ is uniformly continuous.\ Then, there exists a simple $\alpha' \in \mathscr{A}$ such that 
\[
C (v, \alpha[v]) \leq C (v, \alpha' [v]) + \delta 
\]
for all $v \in \N$.
\end{lemma} 
\begin{proof}
 Let $Q_n$ be as in Lemma~\ref{SA: partition} and set $\alpha_n [ v ] := Q_n (\alpha [ v ])$ for $n\in \mathbb N$.\ 
Using the notation from Lemma~\ref{SA: partition}, we observe that 
\begin{align*}
\pi_t \alpha_n [ v ] = \sum_{j = 1}^{N_n} \langle \alpha [ v ], e^n_j \rangle_{\M} \, \pi_t u^n_j &= \sum_{j = 1}^{N_n} \langle \pi_{t^n_j} \alpha [ v ], e^n_j \rangle_{\M}\, \pi_t u^n_j \1_{\{t^n_j < t\}},
\end{align*}
which shows that $\alpha_n$ is simple. 
Furthermore, since $Q_n$ is a contraction, we have $\| \alpha_n [ v ] \|_\M \leq \| \alpha [ v ] \|_\M$. 
 Denoting the modulus of continuity of $\varphi$ by $w_\varphi$, we obtain from Lemma~\ref{lem: sol map}~(c) and part (ii) of Lemma~\ref{SA: partition} that 
	\begin{align*}
		C (v, \alpha [ v ]) &\leq \varphi \big(\G (v, \alpha [ v ])\big) - \frac{1}{2} \| \alpha_n [ v ] \|^2_\M 
		\\&\leq C (v, \alpha_n [v ]) + w_\varphi \big(\| \G (v, \alpha [ v ]) - \G (v, \alpha_n [ v ]) \|_\infty\big) 
		\\&\leq C (v, \alpha_n [v]) + w_\varphi \big(e^L \| \alpha [ v ] - \alpha_n [ v ] \|_\infty \big).
	\end{align*}
	Using that, by Lemma~\ref{SA: partition},
	\[
	\sup_{v \in \N} \| \alpha_n [ v ] - \alpha [ v ] \|_\infty \leq \sup_{ \substack{u \, \in \, \M \\ \| u \|_{\M} \leq R }} \| Q_n u - u \|_\infty \to 0, \quad n \to \infty,
	\] 
	 with $R:=\sup_{v\in \N}\|\alpha[v]\|_\M$, we may take \(n\) large enough such that 
	\[
	C (v, \alpha [ v ]) \leq C (v, \alpha_n [v]) + \delta. \qedhere
	\] 
\end{proof}

For $\delta > 0$, take $\alpha \in \mathscr{A}$ such that 
\begin{align*}
    \lv \leq \inf_{v \in \N} C (v, \alpha [v]) + \delta.
\end{align*}
Since $\varphi$ is bounded and by definition of the cost $C$,
\[
-\|\varphi\|_\infty\leq \lv\leq C (v, \alpha [v]) + \delta\leq \|\varphi\|_\infty+\delta-\frac12\|\alpha[v]\|^2 \quad \text{for all }v\in \N, 
\]
implying
\[
\frac12\|\alpha[v]\|^2\leq 2\|\varphi\|_\infty+\delta\quad \text{for all }v\in \N,
\]
 so that $\alpha$ is bounded by $\sqrt{4\|\varphi\|_\infty+2\delta}$.\ By virtue of Lemma~\ref{lem: piecewise constant approximation 1}, we may therefore assume that $\alpha$ is simple, i.e., 
\begin{align} \label{eq: alpha step}
\pi_t \alpha [ v ] = \sum_{k = 0}^{N-1} g_k [v] \pi_t u^k \1_{\{t_k < t\}}, \quad \text{for all }v \in \N,
\end{align} 
with $N\in \mathbb N$, $0 = t_0 < t_1 < \dots < t_N = 1$, $u^k\in \M$, and \(g_k[ v]\in \bR\) satisfying \eqref{eq: nonanti gk}.\ Now, take a strategy $a \in \A$, and define $h^{a} \in \mathbb{H}$ in the following way: On $[t_0, t_1]$, set 
\[
h^a := \alpha [\lambda_0] \quad\text{and}\quad v^a := a \big(\sqrt{\varepsilon} \, \X + h^a\big),
\] 
where $\lambda_0$ is an arbitrary constant strategy from $\N$.
Inductively, on $(t_n, t_{n + 1}]$ for $n=1,\ldots, N-1$, set 
\[
h^a := \alpha \big[ v^a_n\big] \quad\text{and}\quad v^a := a \big(\sqrt{\varepsilon} \, \X + h^a\big), 
\] 
where $v^a_n = v^a$ on $[0, t_n]$ and $\lambda_0$ otherwise.
Notice that $v^a$ takes values in $\N$.\ Moreover, by Galmarino's test for the filtration $(\cF_t)_{t \in [0, 1]}$ as given by \cite[Lemma~2.5]{NVH}, it is readily seen that $\pi_t h^a$ is $\cF_t$-measurable for every $t \in [0, 1]$. More precisely, for every $t \in [0, 1]$ and \(\omega, \omega' \in \Omega\), the $(\cF_t)_{t \in [0, 1]}$-adaptedness of $a$ shows that $v^a (\omega) = v^a (\omega')$ on $[0, t]$ and hence, by \eqref{eq: nonanti gk}, also $\pi_t h^a(\omega) = \pi_t h^a (\omega')$, whenever $\omega = \omega'$ on $[0, t]$. As $\pi_t h^a$ is clearly $\cF \vee \mathcal{N}$-measurable, this proves that $\pi_t h^a$ is $\cF_t$-measurable.
Let $w_\varphi$ be a bounded modulus of continuity for the bounded and uniformly continuous cost function $\varphi$.\
We deduce from Proposition~\ref{prop: BD formula} and Lemma~\ref{lem: sol map}~(c) that 
\begin{align*}
	V^\varepsilon &\geq \inf_{a \in \mathbb{A}}  \E \bigg[ \varphi \Big( \G \big(a \big(\sqrt{\varepsilon} \, \X + h^a\big), \sqrt{\varepsilon} \, \X + h^a\big) \Big) - \frac{1}{2} \| h^a \|^2_{\M}\bigg]
	\\ &=  \inf_{a \in \mathbb{A}}  \E \bigg[ \varphi \Big( \G \big(v^a, \sqrt{\varepsilon} \, \X +  \alpha [ v^a ]\big) \Big) - \frac{1}{2} \big\| \alpha [v^a] \big\|^2_{\M}\bigg]
	\\&\geq \inf_{a \in \mathbb{A}} \E \bigg[ \varphi \Big( \G \big(v^a, \alpha [ v^a ]\big) \Big) - \frac{1}{2} \big\| \alpha [v^a] \big\|^2_{\M}\bigg] \\&\hspace{3cm}- \sup_{a \in \mathbb{A}} \E \Big[ w_\varphi \Big( \| \G \big(v^a, \sqrt{\varepsilon} \, \X + \alpha [ v^a]\big) - \G \big( v^a, \alpha [ v^a ] \big) \|_\infty \Big) \Big]
	\\&\geq \inf_{v \in \N} C (v, \alpha [v]) - \E \big[ w_\varphi \big( \sqrt{\varepsilon} e^L \| \X \|_\infty \big) \big]
	\\&\geq \lv - \delta - \E \big[ w_\varphi \big( \sqrt{\varepsilon} e^L \| \X \|_\infty \big) \big].
\end{align*}
Taking \(\delta \downarrow 0\) shows that 
\[
V^\varepsilon \geq \lv - \E \big[ w_\varphi \big( \sqrt{\varepsilon} e^L \| \X \|_\infty \big) \big].
\] 
As the final term converges to zero as \(\varepsilon \downarrow 0\), by the dominated convergence theorem, we conclude that the first inequality in \eqref{eq: main} holds. \qed

\subsection{Proof of the last inequality in \eqref{eq: main}}\label{sec:last.ineq}
As a preparatory step, we first define a Borel map $Z \colon [0, 1] \times \N \to \Lambda$ such that $Z (t, v) = v_t$ for a.e.\ $t \in [0, 1]$, and $v = v'$ a.e.\ on $[0, t]$ implies that \(Z (s, v) = Z (s, v')\) for all $s \in [0, t]$. Since $\Lambda$ is assumed to be Polish, it is Borel isomorphic to a Borel subset of $[0, 1]$, i.e., there exists a Borel set $G \subset [0, 1]$ and a Borel bijection $A \colon \Lambda \to G$ with Borel inverse $A^{-1} \colon G \to \Lambda$, see \cite[Theorem~1.8]{kallenberg}. For $n \in \mathbb{N}$, define $S_n \colon [0, 1] \times \N \to \bR$ by 
\[
S_n (t, v) := n \, \int_{(t - 1/n)^+}^t A (v_s) \, ds \quad \text{for $t \in [0, 1]$ and \(v \in \N\).}
\]
We notice that the map $S_n$ is Borel.\ To show this, define $v \mapsto \nu_{v}$ by
\[
\N \to \mathcal{P} ([0, 1] \times \Lambda), \qquad \int_{[0, 1] \time \Lambda} g (s, \lambda) \, \nu_{v} (ds, d \lambda) := \int_0^1 g (s, v_s) \, ds \quad \text{for all } g \in C_b ([0, 1] \times \Lambda), 
\]
where $\mathcal{P} ([0, 1] \times \Lambda)$ denotes the set of all Borel probability measures on $[0, 1] \times \Lambda$. The map $v \mapsto \nu_v$ is continuous by the dominated convergence theorem and hence, $S_n$ is Borel by \cite[Theorem~8.10.61]{bogachev}. 
Define
\[
F := \left \{ (t, v) \in [0, 1] \times \N : \lim_{n \to \infty} S_n (t, v) \text{ exists} \right\}, 
\]
which is a Borel set by \cite[Lemma~1.11]{kallenberg}. Now, set 
\[
S (t, v) := \begin{cases} \displaystyle\lim_{n\to\infty} S_n(t, v), & (t, v) \in F, \\ \ell, & \text{otherwise}, \end{cases}
\]
where $\ell \in G$ is arbitrary, and finally, for $t \in [0, 1]$ and $v \in \N$, 
\[
Z (t, v) := \begin{cases} A^{-1} (S (t, v)), & S (t, v) \in G, \\ A^{-1} (\ell), &\text{otherwise}. \end{cases} 
\]
The map $Z$ is Borel and $Z (t, v) = v_t$ for a.e.\ $t \in [0, 1]$. Furthermore, as $v = v'$ a.e.\ on $[0, t]$ implies that $S_n (s, v) = S_n (s, v')$ for all \(s \in [0, t]\), it also holds that $Z (s, v) = Z (s, v')$ for all $s \in [0, t]$.

We are now in a position to prove the last inequality in \eqref{eq: main}.
Fix $\delta > 0$ and take $\beta \in \mathscr{B}$ such that 
\[
\sup_{u \in \M} C (\beta[u], u) \leq \uv + \delta. 
\] 
With $P_n$ from Lemma~\ref{lem:projection}, set 
\[
a^n_t (\omega) := Z \big(t, \beta [\sqrt{\varepsilon} \, P_n (\omega)] \big), \quad \text{for \(t \in [0, 1]\) and \(\omega \in \Omega\)}.
\]
We now explain that $a^n$ is $(\cF_t)_{t \in [0, 1]}$-progressively measurable. Take $t \in [0, 1]$. By the non-anticipation property of $\beta$, we observe that $\beta [\sqrt{\varepsilon} \, P_n (\omega)] = \beta [\sqrt{\varepsilon} \, \pi_t P_n (\omega)]$ a.e.\ on $[0, t]$. Thus, by the properties of $Z$, we also get that $a^n_t = Z \big(t, \beta [\sqrt{\varepsilon} \, \pi_t P_n] \big)$.\ 
Now, recalling Lemma~\ref{lem:projection}~(i), as $(t, \omega) \mapsto \pi_t P_n (\omega)$ is an $(\cF_t)_{t \in [0, 1]}$-adapted process with continuous paths, it is $(\cF_t)_{t \in [0, 1]}$-progressively measurable. Thus, as \(Z\) is Borel, it follows that $a^n$ is $(\cF_t)_{t \in [0, 1]}$-progressively measurable as well. In summary, $a^n \in \mathbb{A}$.\ We deduce from Proposition~\ref{prop: BD formula} that 
\begin{align*}
V^\varepsilon &\leq \sup_{h \in \mathbb{H}} \E \bigg[ \varphi \Big( \G \big(a^n \big( \X + h / \sqrt{\varepsilon}\big), \sqrt{\varepsilon} \, \X + h\big) \Big) - \frac{1}{2} \| h \|^2_{\M} \bigg].
\end{align*}
Let \(\mathbb{H}_R := \big\{ h \in \mathbb{H} : \E [ \| h \|^2_{\M} ] \leq R \big\}\) with \(R := 4 \| \varphi \|_\infty\).\ We observe that 
\begin{align*}
\sup_{h \in \mathbb{H}} \E \bigg[ \varphi \Big( \G \big(a^n \big( \X &+ h / \sqrt{\varepsilon}\big), \sqrt{\varepsilon} \, \X + h\big) \Big) - \frac{1}{2} \| h \|^2_{\M} \bigg] 
\\&= \sup_{h \in \mathbb{H}_R} \E \bigg[ \varphi \Big( \G \big(a^n \big( \X + h / \sqrt{\varepsilon}\big), \sqrt{\varepsilon} \, \X + h\big) \Big) - \frac{1}{2} \| h \|^2_{\M} \bigg].
\end{align*}
Indeed, if \(h \in \mathbb{H}\) is such that \(\E [ \| h \|^2_{\M} ] > R\), then 
\[
\E \bigg[ \varphi \Big( \G \big(a^n \big( \X + h / \sqrt{\varepsilon}\big), \sqrt{\varepsilon} \, \X + h\big) \Big) - \frac{1}{2} \| h \|^2_{\M} \bigg] < \| \varphi \|_\infty - \frac{R}2 = - \| \varphi \|_\infty, 
\] 
but, with \(h \equiv 0\), 
\[
\E\Big[ \varphi \Big( \G \big(a^n ( \X) , \sqrt{\varepsilon} \, \X \big)\Big) \Big] \geq - \| \varphi\|_\infty, 
\] 
which shows that all \(h \not \in \mathbb{H}_R\) are ignored by the optimization.\ 
Let \(w_\varphi\) be a bounded concave modulus of continuity for the bounded and uniformly continuous cost function \(\varphi\).\ By Lemma~\ref{lem:projection}~(ii), there exists a \(\P\)-null set \(N\) such that, for all \(\omega \not \in N\), Lebesgue a.e.  
\begin{align*}
    a^n \left( \omega + \frac{h (\omega)}{\sqrt{\varepsilon}} \right) = \beta \left [ \sqrt{\varepsilon} \, P_n\left( \omega + \frac{h (\omega)}{\sqrt{\varepsilon}} \right)\right] = \beta \left[ \sqrt{\varepsilon}\, P_n (\omega) + Q_n (h (\omega)) \right].
\end{align*}
By Lemma~\ref{lem: sol map}~(c), this yields that 
\begin{align*}
V^\varepsilon &\leq  
\sup_{h \in \mathbb{H}_R} \E \bigg[ \varphi \Big( \G \big(\beta \big[\sqrt{\varepsilon} P_n  + Q_n (h) \big] , \sqrt{\varepsilon} \, \X + h\big) \Big) - \frac{1}{2} \| h \|^2_{\M} \bigg]  
\\ &\leq \sup_{h \in \mathbb{H}_R} \E \bigg[ \varphi \Big( \G \big(\beta \big[\sqrt{\varepsilon} P_n  + Q_n (h) \big] , \sqrt{\varepsilon} P_n + Q_n (h)\big) \Big) - \frac{1}{2} \| h \|^2_{\M} \bigg] \\&\hspace{2.875cm} + \sup_{h \in \mathbb{H}_R} \E \Big[ w_\varphi \big( e^L (\sqrt{\varepsilon} \| \X - P_n \|_\infty + \|h - Q_n (h)\|_\infty \big) \Big]
\\&\leq \sup_{u \in \M} C ( \beta [ u ], u ) + \sup_{h \in \mathbb{H}_R} \E \Big[ w_\varphi \big( e^L (\sqrt{\varepsilon} \| \X - P_n \|_\infty + \|h - Q_n (h)\|_\infty \big) \Big]
\\&\hspace{2.875cm} + \frac{1}{2} \, \sup_{h \in \mathbb{H}_R} \E\Big[ \big\|\sqrt{\varepsilon} P_n + Q_n (h) \big\|^2_{\M} - \| h \|^2_{\M} \Big]
\\&\leq \uv + \delta + \E \Big[ w_\varphi \big( e^L \sqrt{\varepsilon} \| \X - P_n  \|_\infty \big) \Big] + \sup_{h \in \mathbb{H}_R} \E\big[ w_\varphi (e^L \|h - Q_n (h) \|_\infty )\big]
\\&\hspace{2.875cm} + \frac{1}{2} \, \sup_{h \in \mathbb{H}_R} \E\Big[ \big\|\sqrt{\varepsilon} P_n + Q_n (h) \big\|^2_{\M} - \| h \|^2_{\M} \Big]. 
\end{align*} 
Recall from Lemma~\ref{SA: partition} that \(Q_n \colon \M \to \M\) is a contraction.\ Using this fact we obtain that 
\begin{align*}
	\|\sqrt{\varepsilon} P_n  + Q_n (h) \|^2_{\M} - \| h \|^2_{\M} &\leq \varepsilon \| P_n \|^2 + 2\sqrt{\varepsilon} \| P_n  \|_{\M} \| h \|_{\M} + \| Q_n (h) \|^2_{\M} - \| h \|^2_{\M}
	\\& \leq \varepsilon \| P_n  \|^2 + 2\sqrt{\varepsilon} \| P_n \|_{\M} \| h \|_{\M}.
\end{align*}
Consequently, recalling also Lemma~\ref{lem:projection}~(iii), we obtain that 
\[
\sup_{h \in \mathbb{H}_R} \E\Big[ \big\|\sqrt{\varepsilon} P_n + Q_n (h) \big\|^2_{\M} - \| h \|^2_{\M} \Big] \leq \varepsilon \E\big[ \| P_n \|^2_{\M} \big] + 4 \sqrt{\varepsilon} \Big( \E \big[ \| P_n \|^2_{\M} \big] + R \Big) \to 0
\] 
as \(\varepsilon \downarrow 0\).
As \(\delta> 0\) was arbitrary, we conclude that 
\begin{align*} 
	\limsup_{\varepsilon \downarrow 0} V^\varepsilon \leq \uv + \sup_{h \in \mathbb{H}_R} \E\big[ w_\varphi \big(e^L \|h - Q_n (h) \|_\infty \big)\big].
\end{align*} 
It remains to explain that the last term can be made arbitrarily small when \(n\) is taken large enough. This is the program for the remainder of this proof. 
Notice that 
\[
\sup_{h \in \mathbb{H}_R} \P ( \| h \|_{\M} > K ) \leq \frac{R}{K^2}
\] 
by Chebyshev's inequality. 
Thus, using that \(w_\varphi\) is bounded, we get that 
\begin{align*}
	\sup_{h \in \mathbb{H}_R} \E\big[ w_\varphi \big(e^L \|h - Q_n (h) \|_\infty \big)\big] & \leq \sup_{h \in \mathbb{H}_R} \E\big[ w_\varphi \big(e^L \|h - Q_n (h) \|_\infty \big) \1_{\{ \| h \|_{\M} \leq  K \}} \big] + \textup{const } \frac{R}{K^2}
	\\&\leq w_{\varphi} \Bigg( e^L \sup_{ \substack{ u \, \in \, \M \\ \| u \|_{\M} \leq K }} \| Q_n u - u \|_\infty \Bigg) + \textup{const } \frac{R}{K^2}.
\end{align*}
Letting first \(n \to \infty\) and then \(K \to \infty\), yields that 
\[
\sup_{h \in \mathbb{H}_R} \E\big[ w_\varphi \big(e^L \|h - Q_n (h) \|_\infty \big)\big] \to 0 \quad\text{as } n \to \infty,
\]
where we used Lemma~\ref{SA: partition}~(ii). 
This proves the last inequality in \eqref{eq: main}. \qed

%\bibliographystyle{siam}
%\bibliography{references}

\begin{thebibliography}{10}

\bibitem{BVLT_20}
{\sc J.~Backhoff-Veraguas, D.~Lacker, and L.~Tangpi}, {\em Nonexponential
  {Sanov} and {Schilder} theorems on {Wiener} space: {BSDEs}, {Schr{\"o}dinger}
  problems and control}, Ann. Appl. Probab., 30 (2020), pp.~1321--1367.

\bibitem{bogachev}
{\sc V.~I. Bogachev}, {\em Measure Theory. {Vol}. {I} and {II}}, Berlin:
  Springer, 2007.

\bibitem{BD_98}
{\sc M.~Bou{\'e} and P.~Dupuis}, {\em A variational representation for certain
  functionals of {Brownian} motion}, Ann. Probab., 26 (1998), pp.~1641--1659.

\bibitem{BD_01}
\leavevmode\vrule height 2pt depth -1.6pt width 23pt, {\em Risk-sensitive and
  robust escape control for degenerate diffusion processes}, Math. Control
  Signals Syst., 14 (2001), pp.~62--85.

\bibitem{Budhiraja2024}
{\sc A.~Budhiraja}, {\em On {Some} {Extensions} of the {Bou{\'e}}-{Dupuis}
  {Variational} {Formula}}.
\newblock Preprint, {arXiv}:2403.01562 [math.{PR}], 2024.

\bibitem{BudhirajaDupuis2000}
{\sc A.~Budhiraja and P.~Dupuis}, {\em A variational representation for
  positive functionals of infinite dimensional {Brownian} motion}, Probab.
  Math. Stat., 20 (2000), pp.~39--61.

\bibitem{BudhirajaDupuis2019}
\leavevmode\vrule height 2pt depth -1.6pt width 23pt, {\em Analysis and
  approximation of rare events. {Representations} and weak convergence
  methods}, vol.~94 of Probab. Theory Stoch. Model., New York, NY: Springer,
  2019.

\bibitem{Budhiraja2025}
{\sc A.~Budhiraja and X.~Song}, {\em Large Deviation Principles for Functionals
  of Fractional Brownian Motions}, Springer Nature Singapore, Singapore, 2025,
  pp.~101--131.

\bibitem{conway}
{\sc J.~B. Conway}, {\em A {C}ourse in {F}unctional {A}nalysis}, vol.~96 of
  Grad. Texts Math., Springer, New York, 2nd~ed., 1990.

\bibitem{CK_25}
{\sc D.~Criens and M.~Kupper}, {\em Representation {Theorems} for {Convex}
  {Expectations} and {Semigroups} on {Path} {Space}}.
\newblock Preprint, {arXiv}:2503.10572 [math.{OC}], 2025.

\bibitem{DU_99}
{\sc L.~Decreusefond and A.~S. {\"U}st{\"u}nel}, {\em Stochastic analysis of
  the fractional {Brownian} motion}, Potential Anal., 10 (1999), pp.~177--214.

\bibitem{DZ_98}
{\sc A.~Dembo and O.~Zeitouni}, {\em Large deviations techniques and
  applications.}, vol.~38 of Appl. Math. (N. Y.), New York, NY: Springer,
  2nd~ed., 1998.

\bibitem{DeuschelStroock1989}
{\sc J.-D. Deuschel and D.~W. Stroock}, {\em Large deviations.}, vol.~137 of
  Pure Appl. Math., Academic Press, Boston, MA etc.: Academic Press, Inc.,
  rev.~ed., 1989.

\bibitem{dupuisellis}
{\sc P.~Dupuis and R.~S. Ellis}, {\em A Weak Convergence Approach to the Theory
  of Large Deviations}, Wiley Ser. Probab. Stat., Chichester: John Wiley \&
  Sons, 1997.

\bibitem{KyFan1953}
{\sc K.~Fan}, {\em Minimax theorems}, Proc. Natl. Acad. Sci. USA, 39 (1953),
  pp.~42--47.

\bibitem{FleSon_06}
{\sc W.~H. Fleming and H.~M. Soner}, {\em Controlled {Markov} processes and
  viscosity solutions}, vol.~25 of Stoch. Model. Appl. Probab., New York, NY:
  Springer, 2nd~ed., 2006.

\bibitem{FS2016}
{\sc H.~F{\"o}llmer and A.~Schied}, {\em Stochastic Finance. {An} Introduction
  in Discrete Time.}, De Gruyter Textb., Berlin: de Gruyter, 4th~ed., 2016.

\bibitem{FreidlinWentzell2012}
{\sc M.~I. Freidlin and A.~D. Wentzell}, {\em Random perturbations of dynamical
  systems}, vol.~260 of Grundlehren Math. Wiss., Berlin: Springer, 3rd~ed.,
  2012.

\bibitem{FrizVictoir2010}
{\sc P.~K. Friz and N.~B. Victoir}, {\em Multidimensional stochastic processes
  as rough paths. {Theory} and applications.}, vol.~120 of Camb. Stud. Adv.
  Math., Cambridge: Cambridge University Press, 2010.

\bibitem{jacod79}
{\sc J.~Jacod}, {\em Calcul Stochastique et Probl{\`e}mes de Martingales},
  vol.~714 of Lect. Notes Math., Springer, Cham, 1979.

\bibitem{JP_22}
{\sc A.~Jacquier and A.~Pannier}, {\em Large and moderate deviations for
  stochastic {Volterra} systems}, Stochastic Processes Appl., 149 (2022),
  pp.~142--187.

\bibitem{J_92}
{\sc M.~R. James}, {\em Asymptotic analysis of nonlinear stochastic
  risk-sensitive control and differential games}, Math. Control Signals Syst.,
  5 (1992), pp.~401--417.

\bibitem{kallenberg}
{\sc O.~Kallenberg}, {\em Foundations of Modern Probability. {In} 2 volumes},
  vol.~99 of Probab. Theory Stoch. Model., Cham: Springer, 3rd~ed., 2021.

\bibitem{kato_84}
{\sc T.~Kato}, {\em Perturbation theory for linear operators}, vol.~132 of
  Grundlehren Math. Wiss., Springer, Cham, 2nd~ed., 1995.

\bibitem{zbMATH06864073}
{\sc R.~C. Kraaij}, {\em Large deviations of the trajectory of empirical
  distributions of {Feller} processes on locally compact spaces}, Ann. Probab.,
  46 (2018), pp.~775--828.

\bibitem{zbMATH07370698}
{\sc M.~Kupper and J.~M. Zapata}, {\em Large deviations built on
  max-stability}, Bernoulli, 27 (2021), pp.~1001--1027.

\bibitem{MRTZ_16}
{\sc J.~Ma, Z.~Ren, N.~Touzi, and J.~Zhang}, {\em Large deviations for
  non-{Markovian} diffusions and a path-dependent eikonal equation}, Ann. Inst.
  Henri Poincar{\'e}, Probab. Stat., 52 (2016), pp.~1196--1216.

\bibitem{Mishura2008}
{\sc Y.~Mishura}, {\em Stochastic calculus for fractional {Brownian} motion and
  related processes.}, vol.~1929 of Lect. Notes Math., Berlin: Springer, 2008.

\bibitem{zbMATH08149100}
{\sc M.~Nendel and A.~Sgarabottolo}, {\em A parametric approach to the
  estimation of convex risk functionals based on {Wasserstein} distance}, Appl.
  Math. Optim., 93 (2026), p.~44.
\newblock Id/No 8.

\bibitem{nua_06}
{\sc D.~Nualart}, {\em Stochastic calculus with respect to fractional
  {Brownian} motion}, Ann. Fac. Sci. Toulouse, Math. (6), 15 (2006),
  pp.~63--78.

\bibitem{Nualart2006}
\leavevmode\vrule height 2pt depth -1.6pt width 23pt, {\em {T}he {Malliavin}
  {C}alculus and {R}elated {T}opics.}, Probab. Appl., Berlin: Springer,
  2nd~ed., 2006.

\bibitem{NualartRascanu2002}
{\sc D.~Nualart and A.~R{\u{a}}{\c{s}}canu}, {\em Differential equations driven
  by fractional {Brownian} motion}, Collect. Math., 53 (2002), pp.~55--81.

\bibitem{NR_00}
{\sc D.~Nualart and C.~Rovira}, {\em Large deviations for stochastic {Volterra}
  equations}, Bernoulli, 6 (2000), pp.~339--355.

\bibitem{NVH}
{\sc M.~Nutz and R.~van Handel}, {\em Constructing sublinear expectations on
  path space}, Stochastic Processes Appl., 123 (2013), pp.~3100--3121.

\bibitem{peng_et_all_PPDE_survey}
{\sc S.~Peng, Y.~Song, and F.~Wang}, {\em Survey on path-dependent {PDEs}},
  Chin. Ann. Math., Ser. B, 44 (2023), pp.~837--856.

\bibitem{PTZ_20}
{\sc D.~Possama{\"{\i}}, N.~Touzi, and J.~Zhang}, {\em Zero-sum path-dependent
  stochastic differential games in weak formulation}, Ann. Appl. Probab., 30
  (2020), pp.~1415--1457.

\bibitem{zbMATH05851950}
{\sc D.~W. Stroock}, {\em Probability theory.\ {An} analytic view}, Cambridge:
  Cambridge University Press, 2nd~ed., 2011.

\bibitem{stroock_APV_V3}
\leavevmode\vrule height 2pt depth -1.6pt width 23pt, {\em Probability theory.\
  {An} analytic view}, Cambridge: Cambridge University Press, 3rd~ed., 2024.

\bibitem{TZ_26}
{\sc S.~Tang and J.~Zhou}, {\em Comparison principle of second order
  path-dependent partial differential equations and application to
  path-dependent stochastic differential games}, SIAM J. Control Optim., 64
  (2026), pp.~2757--2783.

\bibitem{Varadhan1984}
{\sc S.~R.~S. Varadhan}, {\em Large deviations and applications}, vol.~46 of
  CBMS-NSF Reg. Conf. Ser. Appl. Math., Philadelphia, PA: Society for
  Industrial {and} Applied Mathematics (SIAM), 1984.

\bibitem{Z_09}
{\sc X.~Zhang}, {\em A variational representation for random functionals on
  abstract {Wiener} spaces}, J. Math. Kyoto Univ., 49 (2009), pp.~475--490.

\end{thebibliography}

\end{document}